\documentclass[12pt]{article}
\usepackage{amsmath}
\usepackage{bbding}
\usepackage{cases}
\usepackage{amsfonts}
\usepackage{mathrsfs}
\usepackage{algorithmic}
\usepackage{algorithm}
\usepackage{amssymb}
\usepackage{graphicx} \usepackage{epstopdf} 
\usepackage{threeparttable} 
\usepackage{enumerate}
\usepackage{datetime}
\usepackage{color}
\usepackage [top=0.8in, bottom=1in, left=2cm, right=2cm]{geometry}
\usepackage{multirow}
\newtheorem{theorem}{Theorem}[section]
\newtheorem{definition}[theorem]{Definition}
\newtheorem{lemma}[theorem]{Lemma}
\newtheorem{corollary}[theorem]{Corollary}
\newtheorem{proposition}[theorem]{Proposition} 
\newtheorem{example}[theorem]{Example}

\newenvironment{proof}{\noindent {\bf Proof.\ }}{$\blacksquare$\vspace{2ex}}

\usepackage{lipsum}

\newcommand\blfootnote[1]{%
  \begingroup
  \renewcommand\thefootnote{}\footnote{#1}%
  \addtocounter{footnote}{-1}%
  \endgroup
}

\begin{document}

\begin{center}

{\LARGE \textbf{ Reduced biquaternion  tensors and applications in color video processing } }
\blfootnote{This research is supported by Macao Science and Technology Development Fund (No. 0013/2021/ITP), the  NSFC (11571220), Canada NSERC, and the joint research and development fund of Wuyi
University, Hong Kong and Macao (2019WGALH20).\par
* Corresponding author
E-mail: xiliu@must.edu.mo
 }
\bigskip
$${\large \textbf{Cui-E Yu}^{1}, \textbf{Xin Liu}^{1, \ast}, \textbf{Hui Luo}^{2}, \textbf{Yang Zhang}^{3}}$$
\newline $^{\text{1}}$Macau Institute of Systems Engineering, Faculty of Innovation Engineering, \\ 
Macau University of Science and Technology,
 Macau, 999078, P. R. China.
 \newline $^{\text{2}}$School of Computer Science and Engineering, Faculty of Innovation Engineering, \\ Macau University of Science and Technology,
 Macau, 999078, P. R. China.
\newline $^{\text{3}}$ Department of Mathematics, University of Manitoba,
\newline Winnipeg, MB, R3T 2N2, Canada.\\

\bigskip
\end{center}

\begin{abstract}
In this paper, we introduce the applications of third-order reduced biquaternion tensors in color video processing. 
We first develop an algorithm for computing the singular value decomposition (SVD) of a third-order reduced biquaternion tensor via a new Ht-product. As theoretical applications, we define the Moore-Penrose inverse of a third-order reduced biquaternion  tensor
and consider  its   characterizations via its SVD.
Using Moore-Penrose inverses,  we mainly discuss the general (or Hermitian) solutions to reduced biquaternion tensor equation $\mathcal{A}\ast_{Ht} \mathcal{X}=\mathcal{B}$   as well as its least-squares solutions.  Finally, we develop two algorithms and apply them in   color video compression and deblurring,  both of which perform better than the compared algorithms.
\\

\\


\end{abstract}

\section{\textbf{Introduction}}

 In 1843, Hamilton extended  the complex number field $\mathbb{C}$ to quaternions which is a four dimensional   division algebra over the real number field $\mathbb{R}$. It is well-known that Hamilton quaternions have successful applications in signal and color image processing (e.g., \cite{ EBS14}). To keep the commutativity of usual number systems, in 1892, Segre (\cite{Segre1892}) discovered the following extension:
\[
\mathbb{H}_{c} =\{a_1+a_2 \mathbf{i} + a_3 \mathbf{j}+a_4 \mathbf{k} \ | \ \mathbf{i}^2=-1, \ \mathbf{j}^2=1, \ \mathbf{ij} = \mathbf{ji}= \mathbf{k}, \  a_1,\dots, a_4 \in \mathbb{R}\}.
\]
Clearly,  $\mathbf{k}^2 = -1, \ \mathbf{ki} =\mathbf{ik} =-\mathbf{j}, \ \mathbf{jk} = \mathbf{kj} = \mathbf{i}$.  $\mathbb{H}_c$ is a four dimensional commutative algebra over $\mathbb{R}$ and contains zero divisors, and  it is often called Segre quaternions (\cite{ Pin10}), commutative quaternions (\cite{kosal201902, pei2004, pei2008}) or reduced biquaternions (\cite{pei2008, schtte1990}). In this paper, we call it ``reduced biquaternions" ( or RB for short).

Reduced biquaternions $\mathbb{H}_c$ has numerous applications in many areas. In signal and
color image processing, Pei et al. (\cite{pei2004, pei2008}) showed that the  operations of the discrete reduced biquaternion Fourier transforms and their  corresponding convolutions  and correlations  were much simpler than the existing implementations of the Hamilton quaternions, and   many questions could be performed simultaneously by $\mathbb{H}_c$. They used reduced biquaternion matrices  to represent the color images, and then compressed the images by the SVDs of reduced biquaternion matrices (\texttt{RBSVDs}).   For comparison, their proposed algorithm (\texttt{RBSVDs}) was faster and has better reconstruction quality than the compared algorithm of computing the SVD  of a quaternion matrix.
By adopting  the special structure of a real representation of a reduced 
biquaternion matrix, Ding et al. \cite{ding2021} investigated the least-squares special minimal norm  solutions to RB matrix equation $AX = B$  and applied the least-squares  minimal norm   RB solution to the color image restoration.

A tensor is an extension of the concept of a matrix, which provides a more general and flexible approach to handle and represent complex multidimensional data structures.  
Tensors  and tensor decompositions (\cite{carroll1970, kilmer2011, kolda2001, kolda2009,  lathauwer2000, qinzhanglp2022, tucker1966})  have been applied  in many areas such as signal processing,
computer science,   data mining,  neuroscience,  etc. 
Studying   tensor rank decomposition may go back to the  1960s: the CANDECOMP/PARAFAC (CP) decomposition (\cite{carroll1970}) and  the Tucker decomposition (\cite{tucker1966}).
In \cite{kilmer2011}, Kilmer and Martin defined a closed multiplication operation (called t-product) for two real tensors,   and  considered some  important theoretical properties and practical tensor approximation problems.
Qin et al. \cite{qinzhanglp2022}  defined a tensor product for third-order quaternion tensors   and  discussed the Qt-SVD of a third-order quaternion tensor with applications in color video. We note that a RB tensor also has good capability in representing color video as it does not lose the correlation between the RGB channels.
 Thus this paper will explore some fundamental theories of RB tensors and their applications, such as utilizing the proposed theories of SVD, Moore-Penrose inverses, the minimal norm least-squares solution to compress or deblur color videos.

This paper is organized as follows. We first introduce some operations over reduced biquaternion tensors, and then we define the Ht-product and explore some  important properties of third-order RB  tensors in Section 2.
In Section 3, we investigate the Ht-SVD of a  third-order RB tensor and develop an algorithm   to compute it.
Two theoretical applications of Ht-SVD are given in Sections 4 and 5. That is, in Section 4,    the Moore-Penrose inverse of a third-order RB  tensor is defined  and  some properties including algebraic expression derived by the Ht-SVD are discussed.  In Section 5, we mainly consider the  general/Hertmian/least-squares solutions/minimal norm least-squares solution to reduced biquaternion tensor equation  $\mathcal{ A} \ast_{Ht} \mathcal{X}=\mathcal{B}$.
Finally, in Section 6, we apply our results to color video compression and deblurring.  The experimental data show that   our methods have  excellent performances comparing with other methods.

\section{ The Ht-product of third-order  tensors over $\mathbb{H}_c$}

Throughout this paper, Euler script letters $(\mathcal{A, B, \cdots}) $ are used to refer to tensors, while the capital letters $(A, B,\cdots)$ represent matrices.
The notations  $\mathbb{R}^{n_1 \times n_2} $  and $\mathbb{R}^{n_1 \times n_2 \times n_3} $ represent the set of all the matrices of dimension $ n_1 \times n_2 $  and all the third-order tensors of dimension $ n_1 \times n_2 \times n_3 $  over algebra $\mathbb{R}$, respectively. For a third-order tensor $ \mathcal{A}$,  we denote  the $i$th horizontal,  the $j$th lateral and the $k$th frontal slice   by   $\mathcal{A}(i,:,:), \mathcal{A}(:,j,:) $ and $\mathcal{A}(:,:,k)$  respectively. For simplicity, let $A^{(k)}$ represent the $k$th frontal slice $\mathcal{A}(:,:,k)$.

Recall that any element  $q =q_0 + q_1\mathbf{i}+ q_2\mathbf{j} + q_3\mathbf{k} \in \mathbb{H}_c$  can   be  expressed as $q = q_a +  \mathbf{j}q_b=q_{(c),1} e_1+q_{(c),2} e_2$, where $q_a = q_0 + q_1\mathbf{i}, q_b = q_2 + q_3\mathbf{i},$
$ q_{(c),1}=q_a+q_b,\ q_{(c),2}=q_a-q_b $,
and $e_1=(1+\mathbf{j})/2, e_2=(1-\mathbf{j})/2 $.
Clearly, $e_1$ and $e_2$
are two special reduced biquaternions  satisfying $e_1=e_1^2=\cdots=e_1^n, e_2=e_2^2=\cdots=e_2^n$ and $e_1 e_2=0$ (\cite{pei2008}). The modulus of  $q$ is $\vert  q \vert = \sqrt{q_0^2 +q_1^2+q_2^2+q_3^2} $. 

For any matrix $A \in \mathbb{H}_c^{m \times n}$, we can write it in two forms, that is,   $A=A_0+A_1 \mathbf{i}+A_2 \mathbf{j}+A_3 \mathbf{k} = B_1+ \mathbf{j} B_2$ with $A_0, \dots, A_3\in \mathbb{R}^{m \times n}, B_1=A_0 +A_1 \mathbf{i},  B_2= A_2 +A_3 \mathbf{i} \in \mathbb{C}^{m \times n}$. Then its conjugate transpose is given by  $A^\ast=A_0^T-A_1^T \mathbf{i}+A_2^T \mathbf{j}-A_3^T \mathbf{k} = B_1^\ast+ \mathbf{j} B_2 ^\ast$, where $B_1^\ast, B_2 ^\ast$ denote the conjugate transpose of complex matrices $B_1, B_2$.  
Moreover,
 $ A \in \mathbb{H}^{n \times n}_c$    is said to be Hermitian (resp. unitary) if and only if  $ A= A^\ast$ (resp. $A A^\ast = A^\ast A=I_n$).
Furthermore,  if we rewrite  $A \in \mathbb{H}_c^{m \times n}$ as $A=A_{(c),1} e_1+ A_{(c),2} e_2$, where $A_{(c),1}=B_1+B_2$ and  $A_{(c),2}=B_1-B_2$ are complex matrices,
then the conjugate transpose of $A$ is $A^\ast=A_{(c),1}^\ast e_1+ A_{(c),2}^\ast e_2 $.
Thus for any  $B = B_{(c),1} e_1+ B_{(c),2} e_2 \in \mathbb{H}_c^{n\times l}$,
$AB=A_{(c),1} B_{(c),1} e_1 + A_{(c),2}  B_{(c),2} e_2$,   and  we can use it to prove that
$(AB)^\ast= B^\ast A^\ast$.
For $A  =(a_{i_1 i_2})\in \mathbb{H}_c^{m\times n}$, we define the Frobenius norm of $A$ as $ \|  A \|_F= \sqrt{  \sum_{i_1 =1}^{m}   \sum_{i_2 =1}^{n}  \vert  a_{i_1 i_2 } \vert ^2   }$.
Moreover, it can be  verified that
$\|  A \|_F^2 = \mbox{tr}(\Re (A^{\ast} A)),$ 
where 
 $\Re(\cdot)$ denotes the real part of a reduced biquaternion matrix.
\vspace{.2cm}

The block circulant matrix $\textbf{circ}(\mathcal{A})$  generated by  a third-order tensor  $\mathcal{A}$'s
frontal slices $A^{(1)},$ $A^{(2)},  \cdots, A^{(n_3)}$ is given as
$$\textbf{circ}(\mathcal{A})= \begin{bmatrix}
 A^{(1)} & A^{(n_3)} & \cdots  & A^{(2)} \\
 A^{(2)} & A^{(1)}   & \cdots  & A^{(3)}\\
 \vdots  & \vdots    &\ddots   &\vdots\\
 A^{(n_3)} & A^{(n_3-1)} & \cdots & A^{(1)}\\
\end{bmatrix}. $$
We will use $\textbf{Vec}$ operator to convert a tensor  $\mathcal{A}$  into a block matrix  $\textbf{Vec}(\mathcal{A}) $:
\begin{equation*} \label{Vec}
\textbf{Vec}(\mathcal{A})=\begin{bmatrix}
 A^{(1)}  \\
 A^{(2)} \\
 \vdots  \\
 A^{(n_3)} \\
\end{bmatrix},
\end{equation*}\
and use the inverse operation $\textbf{Fold}$  to take the block  matrix $\textbf{Vec}(\mathcal{A})$  back to the tensor $\mathcal{A}$:
\[
\textbf{Fold}(\textbf{Vec}(\mathcal{A}))=\mathcal{A}.
\]
Moreover, we use the frontal slices of $\mathcal{A}$ to define a block diagonal matrix as following:
\begin{equation}
\label{hatdiag}
 \textbf{diag}({\mathcal{A}}) \doteq \begin{bmatrix}
 {A}^{(1)}   &  &  &  \\
     & {A}^{(2)}  &  & \\
     &   &\ddots   & \\
     &   &  & {A}^{(n_3)}\\
\end{bmatrix}.
\end{equation}

Note that by using the normalized Discrete Fourier transformation (DFT) matrix, a  complex circulant matrix can be diagonalized   (\cite{kilmer2011}, \cite{qinzhanglp2022}).  Suppose $\mathcal{A} \in \mathbb{C}^{n_1\times n_2 \times n_3}$  is a third-order complex tensor,
then
\begin{equation}
\label{circdiag}
  (F_{n_3} \otimes I_{n_1})  \textbf{circ}(\mathcal{A})  (F_{n_3}^\ast \otimes I_{n_2})=\begin{bmatrix}
 \widehat{A}^{(1)}   &  &  &  \\
     & \widehat{A}^{(2)}  &  & \\
     &   &\ddots   & \\
     &   &  & \widehat{A}^{(n_3)}\\
\end{bmatrix} = \textbf{diag}(\widehat{\mathcal{A}}) \in \mathbb{C}^{n_3n_1 \times n_3n_2},
\end{equation}
where  $\otimes$ denotes the Kronecker product, $\widehat{A}^{(i)}$  ($i=1, \cdots, n_3$) are the frontal slices of the tensor $\widehat{\mathcal{A}} \in \mathbb{C}^{n_1\times n_2 \times n_3}$  which is the result of  DFT of $\mathcal{A}$ along the third mode,
 $F_{n_3}$ is an $n_3 \times n_3$ unitary matrix defined by  $F_{n_3}(s,t) = \frac{1}{ \sqrt{n_3} } \omega ^{(s-1)(t-1)}, \ \omega = exp(-\frac{2\pi \mathbf{i}}{n_3} )$.

\begin{definition}
For the tensor $\mathcal{A}\in \mathbb{H}_c^{n_1\times n_2 \times n_3}$, the DFT 
$\widehat{\mathcal{A}}$ of $\mathcal{A} \in \mathbb{H}_c^{n_1\times n_2 \times n_3} $ along the third mode is defined as     
\[
\widehat{\mathcal{A}}(s,t,:)=\sqrt{n_3} F_{n_3} \mathcal{A}(s,t,:).
\]
\end{definition}
As we can see that after the DFT, the resulting vector $\widehat{\mathcal{A}}(s,t,:)$ is  still a vector in $\mathbb{H}^{n_3}_c$.
We can verify the relation between $\mathcal{A}$ and $\widehat{\mathcal{A}}$ as follows.
\begin{equation}\label{hatA}
 \textbf{Vec}(\widehat{\mathcal{A}}) =\begin{bmatrix}
 \widehat{A}^{(1)}  \\
 \widehat{A}^{(2)} \\
 \vdots  \\
 \widehat{A}^{(n_3)} \\
\end{bmatrix}= \sqrt{n_3} (F_{n_3} \otimes I_{n_1}) \textbf{Vec}(\mathcal{A}).
\end{equation}

Now we extend $t$-product of two third-order real tensors  defined in \cite{kilmer2011} to Reduced biquaternions.

\begin{definition}(Ht-product)\label{defHt}
Let $\mathcal{A}=\mathcal{A}_{1,\mathbf{i}}+\mathbf{j} \mathcal{A}_{\mathbf{j},\mathbf{k}} \in \mathbb{H}_c^{n_1\times n_2 \times n_3}$ and $\mathcal{B} \in \mathbb{H}_c^{n_2\times n_4 \times n_3}$.   Then $Ht$-product of $\mathcal{A} $ and $\mathcal{B} $ is defined as
 \[
 \mathcal{A}\ast_{Ht} \mathcal{B}= \textbf{Fold}((\textbf{circ}(\mathcal{A}_{1,\mathbf{i}})+\mathbf{j} \textbf{circ}(\mathcal{A}_{\mathbf{j},\mathbf{k}}))  \textbf{Vec}(\mathcal{B}))
 \in \mathbb{H}_c^{n_1\times n_4 \times n_3}.
 \]
\end{definition}

  The addition and scalar multiplication rules of tensors over $\mathbb{H}_c$ are  defined in a usual way.

\begin{theorem}
\label{Htthm}
Let $\mathcal{A} \in \mathbb{H}_c^{n_1\times n_2 \times n_3}, \ \mathcal{B} \in \mathbb{H}_c^{n_2\times n_4 \times n_3}$ and $\mathcal{C} \in \mathbb{H}_c^{n_1\times n_4 \times n_3}$ with the corresponding DFTs $\widehat{\mathcal{A}}, \  \widehat{\mathcal{B}}$ and $ \widehat{\mathcal{C}}$.
Then $\mathcal{A}\ast_{Ht} \mathcal{B}=\mathcal{C}$  if and only if
$\textbf{diag}(\widehat{\mathcal{A}})  \textbf{diag}(\widehat{\mathcal{B}})= \textbf{diag}(\widehat{\mathcal{C}}).$
\end{theorem}

\begin{proof}
Since $\widehat{\mathcal{A}}$ is the DFT of $ \mathcal{A}=\mathcal{A}_{1,\mathbf{i}}+ \mathbf{j} \mathcal{A}_{\mathbf{j},\mathbf{k}}$ along the third mode,
it follows from (\ref{hatA}) and $F_{n_3} \mathbf{j}=\mathbf{j}F_{n_3}$  that
\begin{align*}
  \textbf{Vec}(\widehat{\mathcal{A}}) & = \sqrt{n_3} (F_{n_3} \otimes I_{n_1}) ( \textbf{Vec}(\mathcal{A}_{1,\mathbf{i}}) +\mathbf{j} \textbf{Vec}(\mathcal{A}_{\mathbf{j},\mathbf{k}}) )\\
   & = \textbf{Vec}(\widehat{\mathcal{A}}_{1,\mathbf{i}})+\mathbf{j}\textbf{Vec}(\widehat{\mathcal{A}}_{\mathbf{j},\mathbf{k}}),
\end{align*}
which implies
\begin{equation}
\label{diag1}
  \textbf{diag}(\widehat{\mathcal{A}})= \textbf{diag}(\widehat{\mathcal{A}}_{1,\mathbf{i}})+\mathbf{j} \textbf{diag}(\widehat{\mathcal{A}}_{\mathbf{j},\mathbf{k}}).
\end{equation}
By Definition \ref{defHt}
and the identities  (\ref{circdiag}) and (\ref{hatA}),  we  have
\begin{align*}
\textbf{Vec}(\mathcal{C})& = \textbf{Vec}(\mathcal{A}\ast_{Ht} \mathcal{B}) \\
& = (\textbf{circ}(\mathcal{A}_{1,\mathbf{i}})+\mathbf{j} \textbf{circ}(\mathcal{A}_{\mathbf{j},\mathbf{k}}))  \textbf{Vec}(\mathcal{B})\\
& =(F_{n_3}^\ast \otimes I_{n_1})[ \textbf{diag}(\widehat{\mathcal{A}}_{1,\mathbf{i}})  + \mathbf{j} \textbf{diag}(\widehat{\mathcal{A}}_{\mathbf{j},\mathbf{k}})  ]  (F_{n_3} \otimes I_{n_2})  \textbf{Vec}(\mathcal{B})\\
& =\frac{1}{\sqrt{n_3}} (F_{n_3}^\ast \otimes I_{n_1})
[ \textbf{diag}(\widehat{\mathcal{A}}_{1,\mathbf{i}})  + \mathbf{j} \textbf{diag}(\widehat{\mathcal{A}}_{\mathbf{j},\mathbf{k}}) ]   \textbf{Vec}(\widehat{\mathcal{B}}).
\end{align*}
Thus
\[
\textbf{Vec}(\widehat{\mathcal{C}})= \sqrt{n_3} (F_{n_3} \otimes I_{n_1}) \textbf{Vec}(\mathcal{C})= [ \textbf{diag}(\widehat{\mathcal{A}}_{1,\mathbf{i}})  + \mathbf{j} \textbf{diag}(\widehat{\mathcal{A}}_{\mathbf{j},\mathbf{k}})  ]   \textbf{Vec}(\widehat{\mathcal{B}}).
\]
Combing with equation (\ref{diag1})  yields
 $\textbf{Vec}(\widehat{\mathcal{C}})= \textbf{diag}(\widehat{\mathcal{A}})  \textbf{Vec}(\widehat{\mathcal{B}})$.
 Therefore,  $$ \textbf{diag}(\widehat{\mathcal{C}})= \textbf{diag}(\widehat{\mathcal{A}})  \textbf{diag}(\widehat{\mathcal{B}}).$$
\end{proof}

According to the above result, we can give an equivalent  definition of  the Ht-product of two third-order reduced biquaternion tensors.
\begin{definition}
\label{def2Ht}
Let $\mathcal{A} \in \mathbb{H}_c^{n_1\times n_2 \times n_3}$ and $\mathcal{B} \in \mathbb{H}_c^{n_2\times n_4 \times n_3}$.   The $Ht$-product of $\mathcal{A} $ and $\mathcal{B} $ is defined as
 \[
 \mathcal{A}\ast_{Ht} \mathcal{B}= {\textbf{Fold}}\left(  \frac{1}{\sqrt{n_3}} (F_{n_3}^\ast \otimes I_{n_1})  \textbf{diag}(\widehat{\mathcal{A}})  \textbf{diag}(\widehat{\mathcal{B}})  (e \otimes I_{n_4})  \right),
 \]
 where $e$ is an $n_3 \times 1$   vector with all the entries equal to 1.
\end{definition}

This Ht-product obeys some usual algebraic operation laws,  such as the associative law and the distributive law.

\begin{proposition}
\label{Htassocia}
Let $\mathcal{A} \in \mathbb{H}_c^{n_1\times n_2 \times n_3}, \ \mathcal{B} \in \mathbb{H}_c^{n_2\times n_4 \times n_3}$ and $\mathcal{C} \in \mathbb{H}_c^{n_4\times n_5 \times n_3}$. Then
\[
(\mathcal{A}\ast_{Ht} \mathcal{B}) \ast_{Ht} \mathcal{C}=\mathcal{A}\ast_{Ht} (\mathcal{B} \ast_{Ht} \mathcal{C}).
\]
\end{proposition}

\begin{proof}
By Definition \ref{def2Ht}, we have
\begin{align*}
  (\mathcal{A}\ast_{Ht} \mathcal{B}) \ast_{Ht} \mathcal{C} = & \textbf{Fold}( \frac{1}{\sqrt{n_3}} (F_{n_3}^\ast \otimes I_{n_1})  \textbf{diag}(\widehat{\mathcal{A}})  \textbf{diag}(\widehat{\mathcal{B}})  (e \otimes I_{n_4})  )   \ast_{Ht} \mathcal{C} \\
  = & \textbf{Fold}( \frac{1}{\sqrt{n_3}} (F_{n_3}^\ast \otimes I_{n_1})  (\textbf{diag}(\widehat{\mathcal{A}})  \textbf{diag}(\widehat{\mathcal{B}}))   \textbf{diag}(\widehat{\mathcal{C}})   (e \otimes I_{n_5})  ),
\end{align*}
and
\begin{align*}
  \mathcal{A}\ast_{Ht} (\mathcal{B} \ast_{Ht} \mathcal{C}) =  & \mathcal{A}\ast_{Ht}  \textbf{Fold}(\frac{1}{\sqrt{n_3}} (F_{n_3}^\ast \otimes I_{n_2})   \textbf{diag}(\widehat{\mathcal{B}}) \textbf{diag}(\widehat{\mathcal{C}})  (e \otimes I_{n_5})    ) \\
  = & \textbf{Fold} (\frac{1}{\sqrt{n_3}} (F_{n_3}^\ast \otimes I_{n_1})    \textbf{diag}(\widehat{\mathcal{A}})   (\textbf{diag}(\widehat{\mathcal{B}}) \textbf{diag}(\widehat{\mathcal{C}}))  (e \otimes I_{n_5})  ).
\end{align*}
Notice that $ \textbf{diag}(\widehat{\mathcal{A}}), \textbf{diag}(\widehat{\mathcal{B}}) $ and $\textbf{diag}(\widehat{\mathcal{C}})$ are matrices over $\mathbb{H}_c$. Hence the associative law holds, and thus the result holds.
\end{proof}

In a similar way as in Proposition \ref{Htassocia}, we have the following distributive laws. For brevity, we omit the proofs.
 \begin{proposition}
\label{Htdistr}
Let $\mathcal{A} \in \mathbb{H}_c^{n_1\times n_2 \times n_3}, \ \mathcal{B}, \  \mathcal{C} \in \mathbb{H}_c^{n_2\times n_4 \times n_3}$    
and  $\mathcal{D} \in \mathbb{H}_c^{n_4\times n_5 \times n_3}$. Then
\[
 \mathcal{A} \ast_{Ht} ( \mathcal{B} + \mathcal{C} ) =\mathcal{A}\ast_{Ht} \mathcal{B} +\mathcal{A} \ast_{Ht} \mathcal{C},
\]
\[
  ( \mathcal{B} + \mathcal{C} ) \ast_{Ht} \mathcal{D} =\mathcal{B}\ast_{Ht} \mathcal{D} +\mathcal{C} \ast_{Ht} \mathcal{D}.
\]
\end{proposition}

 Suppose $\mathcal{A} \in \mathbb{C}^{n_1 \times n_2 \times n_3}$. Using a similar approach in \cite{qinzhanglp2022}, we calculate the conjugate transpose $\mathcal{A}^\ast \in \mathbb{C}^{n_2 \times n_1 \times n_3}$ of $\mathcal{A}$
 as follows:  conjugately transposing each frontal slice of $\mathcal{A}$ first, and   reversing the order of conjugately transposed frontal slices $2$ through $n_3$ next.

\begin{definition}
\label{defHtconjtr}
The conjugate transpose $\mathcal{A}^\ast $ of  $\mathcal{A}= \mathcal{A}_{1,\mathbf{i}}+\mathbf{j} \mathcal{A}_{j,k} \in \mathbb{H}_c^{n_1\times n_2 \times n_3}$ with $\mathcal{A}_{1,\mathbf{i}},  \mathcal{A}_{\mathbf{j},\mathbf{k}} \in \mathbb{C}^{n_1 \times n_2 \times n_3}$ is  defined as   $\mathcal{A}^\ast= \mathcal{A}_{1,\mathbf{i}}^\ast+j \mathcal{A}_{\mathbf{j},\mathbf{k}}^\ast$ with $\mathcal{A}_{1,\mathbf{i}}^\ast, \mathcal{A}_{\mathbf{j},\mathbf{k}}^\ast \in \mathbb{C}^{n_2 \times n_1 \times n_3}$.
\end{definition}

It is not difficult to verify that $(\mathcal{A}^\ast)^\ast=\mathcal{A}$ and $(\mathcal{A}+ \mathcal{B} )^\ast=\mathcal{A}^\ast +\mathcal{B}^\ast $ for $\mathcal{A}, \mathcal{B}  \in \mathbb{H}_c^{n_1\times n_2 \times n_3}$.

\begin{definition}\label{defHtident}
The   identity   tensor  $\mathcal{I}_{nnm} \in \mathbb{H}_c^{n\times n\times m}$  (simply denoted by  $\mathcal{I}$) is the tensor in which the first frontal slice is $I_n$  and the other slice matrices are zero.
\end{definition}

We can obtain that $\mathcal{I} \ast_{Ht} \mathcal{I}=\mathcal{I}$,\  $\mathcal{I}^\ast=\mathcal{I}$  and $\mathcal{A}  \ast_{Ht}   \mathcal{I}  = \mathcal{I} \ast_{Ht} \mathcal{A}=  \mathcal{A}.$

\begin{definition}
\label{defHtinv}
A   tensor  $\mathcal{A} \in \mathbb{H}_c^{n \times n \times m}$ is called  invertible if there is a  tensor  $\mathcal{B}  \in \mathbb{H}_c^{n \times n \times m}$ such that
$ \mathcal{A} \ast_{Ht} \mathcal{B}=\mathcal{I}$  and
$  \mathcal{B} \ast_{Ht} \mathcal{A}=\mathcal{I}$.
\end{definition}

By using  Proposition \ref{Htassocia}  and  the fact $\mathcal{A}  \ast_{Ht}   \mathcal{I}  = \mathcal{I} \ast_{Ht} \mathcal{A}=  \mathcal{A}$, we have the following result.

\begin{proposition}
\label{Htinv}
 The inverse of  a  reduced  biquaternion tensor $\mathcal{A} \in \mathbb{H}_c^{n \times n \times m}$  is unique if it exists, denoted by $\mathcal{A}^{-1}$.\end{proposition}


\begin{definition}
\label{defHtidemp}
Let $\mathcal{A} \in \mathbb{H}_c^{n \times n \times m}$.  $\mathcal{A} $ is called idempotent  if
$ \mathcal{A} \ast_{Ht} \mathcal{A}=\mathcal{A}$;
  $ \mathcal{A} $ is called unitary (resp. Hermitian) if
$ \mathcal{A} \ast_{Ht} {\mathcal{A}}^\ast={\mathcal{A}}^\ast  \ast_{Ht} \mathcal{A} =\mathcal{I}  \ (resp. \ \mathcal{A}={\mathcal{A}}^\ast ).$
\end{definition}

\begin{lemma}
\label{Ctp01}
Let $\mathcal{A} \in \mathbb{C}^{n_1\times n_2 \times n_3}$. Then $ \textbf{diag}(  \widehat{ {\mathcal{A}}^\ast  })= (\textbf{diag}(\widehat{\mathcal{A}}))^\ast . $
\end{lemma}
\begin{proof}
It is easy to see   that  $ \textbf{circ}({\mathcal{A}}^\ast )= (\textbf{circ}({\mathcal{A}}))^\ast$ for $\mathcal{A} \in \mathbb{C}^{n_1\times n_2 \times n_3}.$
According to the identity (\ref{circdiag}), we have
\begin{align*}
  \textbf{diag}(\widehat{ {\mathcal{A}}^\ast  })  & =(F_{n_3} \otimes I_{n_2})  \textbf{circ}({\mathcal{A}}^\ast ) (F_{n_3}^\ast \otimes I_{n_1}) \\
   & =(F_{n_3} \otimes I_{n_2})  (\textbf{circ}({\mathcal{A}}))^\ast   (F_{n_3}^\ast \otimes I_{n_1}) \\
   &= ( \textbf{diag}(\widehat{ \mathcal{A} }) )^\ast.
\end{align*}
\end{proof}

The above result over $\mathbb{C}$ can be extended to $\mathbb{H}_c$ as follows.

\begin{proposition}
\label{Htp02}
Let $\mathcal{A} \in \mathbb{H}_c^{n_1\times n_2 \times n_3}$. Then $ \textbf{diag}(  \widehat{ {\mathcal{A}}^\ast  })= (\textbf{diag}(\widehat{\mathcal{A}}))^\ast . $
\end{proposition}
\begin{proof}
Since  $ \widehat{ {\mathcal{A}}^\ast  } $ is the  DFT of ${\mathcal{A}}^\ast $ and $F_{n_3} \mathbf{j}=\mathbf{j}F_{n_3}$, it follows from (\ref{hatA}) that
\begin{align*}
  \textbf{Vec}(\widehat{ {\mathcal{A}}^\ast  }) & = \sqrt{n_3} (F_{n_3} \otimes I_{n_2}) \textbf{Vec}({\mathcal{A}}^\ast) \\
   & = \sqrt{n_3} (F_{n_3} \otimes I_{n_2})  ( \textbf{Vec}({\mathcal{A}}_{1,\mathbf{i}}^\ast) + \mathbf{j}  \textbf{Vec}({\mathcal{A}}_{\mathbf{j},\mathbf{k}}^\ast) ) \\
   & = \textbf{Vec}(\widehat{ {\mathcal{A}}_{1,\mathbf{i}}^\ast }) + \mathbf{j}  \textbf{Vec}(\widehat{ {\mathcal{A}}_{\mathbf{j},\mathbf{k}}^\ast }),
\end{align*}
which is equivalent to
\begin{equation} \label{diaeq}
\textbf{diag}(\widehat{ {\mathcal{A}}^\ast  })=  \textbf{diag}(\widehat{ {\mathcal{A}}_{1,\mathbf{i}}^\ast }) + \mathbf{j}  \textbf{diag}(\widehat{ {\mathcal{A}}_{\mathbf{j},\mathbf{k}}^\ast }).
\end{equation}
Similarly, we have $\textbf{diag}(\widehat{ \mathcal{A}  })=  \textbf{diag}(\widehat{ {\mathcal{A}}_{1,\mathbf{i}} }) + \mathbf{j}  \textbf{diag}(\widehat{ {\mathcal{A}}_{\mathbf{j},\mathbf{k}} }) \in \mathbb{H}_c^{n_3n_1\times n_3n_2 }.$  Its conjugate transpose is   $( \textbf{diag}(\hat{ \mathcal{A}  }))^\ast= (\textbf{diag}(\widehat{ {\mathcal{A}}_{1,\mathbf{i}} }))^\ast+ \mathbf{j} (\textbf{diag}(\widehat{ {\mathcal{A}}_{\mathbf{j},\mathbf{k}} })) ^\ast .$  Combing with the identity  (\ref{diaeq}) and Lemma \ref{Ctp01}, we have  $  \textbf{diag}(\widehat{ {\mathcal{A}}^\ast  })=(\textbf{diag}(\widehat{ \mathcal{A} }))^\ast.$
\end{proof}

\begin{corollary}
\label{cor1}
$\mathcal{A} \in \mathbb{H}_c^{n \times n \times m}$   is Hermitian (resp. unitary) if and only if   $\textbf{diag}(\widehat{ \mathcal{A} })$ is Hermitian (resp. unitary).
\end{corollary}

\begin{proof}
By  Definition \ref{defHtident},
it is clear that  $\textbf{diag}(\widehat{ \mathcal{I}})=I_{nm}.$
Then, by Theorem \ref{Htthm} and Proposition \ref{Htp02}, the result follows.
\end{proof}

\begin{proposition}
\label{Htconjtr}
Given  $\mathcal{A} \in \mathbb{H}_c^{n_1\times n_2 \times n_3}$ and $\mathcal{B} \in \mathbb{H}_c^{n_2\times n_4 \times n_3}$,   we have $  (\mathcal{A}\ast_{Ht}  \mathcal{B})^\ast = \mathcal{B}^\ast \ast_{Ht} \mathcal{A}^\ast. $
\end{proposition}
\begin{proof}
From  Theorem \ref{Htthm} and Proposition \ref{Htp02}, we only need to show
$$
( \textbf{diag}(\widehat{\mathcal{A}})  \textbf{diag}(\widehat{\mathcal{B}})  )^\ast =
(\textbf{diag}(\widehat{\mathcal{B}}))^\ast (\textbf{diag}(\widehat{\mathcal{A}}))^\ast.
$$
In fact, this equality holds for matrices $\textbf{diag}(\widehat{\mathcal{A}})$ and $\textbf{diag}(\widehat{\mathcal{B}})$  over $\mathbb{H}_c$. 
\end{proof}

\section{The Ht-SVD of third-order  tensors over $\mathbb{H}_c$ }
\label{s3}

The  SVD of a matrix over $\mathbb{H}_c$ has been discussed in some papers (see, e.g., \cite{pei2008}).
Let $A = A_{(c),1} e_1+ A_{(c),2} e_2 \in \mathbb{H}_c^{m\times n}$.  If $A_{(c),1}$ and $ A_{(c),2}$ have the following SVDs:
\begin{center}
$A_{(c),1}=U_1 \Sigma_1 V_1^\ast $ and $A_{(c),2}=U_2 \Sigma_2 V_2^\ast,$
\end{center}
then the SVD of     $A$ is
\begin{equation}
\label{RBSVD}
A= U \Sigma V^\ast,\end{equation}
 where $U=U_1 e_1+U_2 e_2$  and $V=V_1 e_1+V_2 e_2$ are unitary matrices, $\Sigma=\Sigma_1 e_1+\Sigma_2 e_2 $ is  diagonal  but not a real matrix in general.  We define  the {\bf rank} of a  matrix $A \in \mathbb{H}_c^{m \times n}$ is the number of its non-zero singular values.

 Recall that a third-order tensor is  called  {\bf $f$-diagonal}  if each frontal slice is diagonal (\cite{kilmer2011}).
 Next,  we propose the Ht-SVD
of a reduced biquaternion tensor.

\begin{theorem}(Ht-SVD)
\label{HtSVD}
Let  $\mathcal{A} \in \mathbb{H}_c^{n_1\times n_2 \times n_3}$. Then there exist unitary  tensors $\mathcal{U}\in \mathbb{H}_c^{n_1 \times n_1  \times n_3} $ and $\mathcal{V}\in \mathbb{H}_c^{n_2 \times n_2  \times n_3} $ such that
 \[
 \mathcal{A}=\mathcal{U}\ast_{Ht} \mathcal{S}\ast_{Ht} \mathcal{V}^\ast,
 \]
where $ \mathcal{S} \in \mathbb{H}_c^{n_1 \times n_2  \times n_3} $ is  $f$-diagonal. Such decomposition is called $Ht$-SVD of $\mathcal{A}$.
\end{theorem}

\begin{proof}
   We first transform $\mathcal{A}$ into   $\widehat{\mathcal{A}}$ by using DFT along the third mode, and then obtain the block diagonal matrix $\textbf{diag}(\widehat{\mathcal{A}})$ by (\ref{hatdiag}).
In addition, we consider the SVD of each $\widehat{A}^{(s)}$ in $\textbf{diag}(\widehat{\mathcal{A}})$ with formula (\ref{RBSVD}).
 Suppose  $\widehat{A}^{(s)}= \widehat{U}^{(s)}  \widehat{\Sigma}^{(s)} (\widehat{V}^{(s)})^\ast, s=1, 2, \cdots, n_3.$
 Then we have
\begin{equation*}\label{hatSVD0}
 \textbf{diag}(\widehat{\mathcal{A}}) =
\begin{bmatrix}
 \widehat{U}^{(1)}   &  &  &  \\
     & \widehat{U}^{(2)}  &  & \\
     &   &\ddots   & \\
     &   &  & \widehat{U}^{(n_3)}\\
\end{bmatrix}
\begin{bmatrix}
 \widehat{\Sigma}^{(1)}   &  &  &  \\
     & \widehat{\Sigma}^{(2)}  &  & \\
     &   &\ddots   & \\
     &   &  & \widehat{\Sigma}^{(n_3)}\\
\end{bmatrix}
\begin{bmatrix}
 (\widehat{V}^{(1)})^\ast   &  &  &  \\
     & (\widehat{V}^{(2)})^\ast  &  & \\
     &   &\ddots   & \\
     &   &  & (\widehat{V}^{(n_3)})^\ast\\
\end{bmatrix}.
\end{equation*}
Setting
\[
\textbf{diag}(\widehat{\mathcal{U}})=\begin{bmatrix}
 \widehat{U}^{(1)}   &  &  &  \\
     & \widehat{U}^{(2)}  &  & \\
     &   &\ddots   & \\
     &   &  & \widehat{U}^{(n_3)}\\
\end{bmatrix}, \quad  \textbf{diag}(\widehat{\mathcal{V}})= \begin{bmatrix}
 \widehat{V}^{(1)}   &  &  &  \\
     & \widehat{V}^{(2)}  &  & \\
     &   &\ddots   & \\
     &   &  & \widehat{V}^{(n_3)}\\
\end{bmatrix},
\]
and $$\textbf{diag}(\hat{\mathcal{S}})= \begin{bmatrix}
 \widehat{\Sigma}^{(1)}   &  &  &  \\
     & \widehat{\Sigma}^{(2)}  &  & \\
     &   &\ddots   & \\
     &   &  & \widehat{\Sigma}^{(n_3)}\\
\end{bmatrix},$$   we have
\begin{equation}\label{hatSVD}
 \textbf{diag}(\widehat{\mathcal{A}})=  \textbf{diag}(\widehat{\mathcal{U}})  \textbf{diag}(\widehat{\mathcal{S}})
( \textbf{diag}(\widehat{\mathcal{V}}) )^\ast,
\end{equation}
in which  $ \textbf{diag}(\widehat{\mathcal{U}})$  and $\textbf{diag}(\widehat{\mathcal{V}})$ are unitary  matrices,
and $\textbf{diag}(\widehat{\mathcal{S}})$ is a block diagonal  matrix with elements being  diagonal matrices in general.
Next,  we construct reduced biquaternion tensors $\mathcal{U}, \mathcal{V}$ and $\mathcal{S}$  as follows:
\[
\mathcal{U} \doteq  \textbf{Fold} \left(\frac{1}{\sqrt{n_3}} (F_{n_3}^\ast \otimes I_{n_1})    \textbf{diag}(\widehat{\mathcal{U}}) (e \otimes I_{n_1}) \right),
\]
\[
\mathcal{V} \doteq  \textbf{Fold} \left( \frac{1}{\sqrt{n_3}} (F_{n_3}^\ast \otimes I_{n_2})    \textbf{diag}(\widehat{\mathcal{V}}) (e \otimes I_{n_2}) \right),
\]
\[
\mathcal{S} \doteq  \textbf{Fold} \left(\frac{1}{\sqrt{n_3}} (F_{n_3}^\ast \otimes I_{n_1})    \textbf{diag}(\widehat{\mathcal{S}}) (e \otimes I_{n_2}) \right).
\]
Observing that $  \textbf{Vec}(\widehat{\mathcal{U}}) = \sqrt{n_3} (F_{n_3} \otimes I_{n_1}) \textbf{Vec}(\mathcal{U}),  \textbf{Vec}(\widehat{\mathcal{V}}) = \sqrt{n_3} (F_{n_3} \otimes I_{n_2}) \textbf{Vec}(\mathcal{V}), \textbf{Vec}(\widehat{\mathcal{S}}) = \sqrt{n_3} (F_{n_3} \otimes I_{n_1}) \textbf{Vec}(\mathcal{S}).$
That is, $ \widehat{\mathcal{U}}, \widehat{\mathcal{V}}, \widehat{\mathcal{S}}$ are the DFT  of tensors $\mathcal{U}, \mathcal{V}, \mathcal{S}$, respectively.
 Thus by applying  Theorem \ref{Htthm} and Proposition \ref{Htp02} , we infer from the equation (\ref{hatSVD}) that
\[
\mathcal{A}=\mathcal{U}\ast_{Ht} \mathcal{S}\ast_{Ht} \mathcal{V}^\ast.
\]
Obviously, $ \mathcal{S}$ is $f$-diagonal. And according to  Corollary \ref{cor1}, we know that $ \mathcal{U}$, $ \mathcal{V}$ are our required unitary tensors.
\end{proof}

Based on Theorem \ref{HtSVD}, we give the reduced  biquaternion tensor SVD decomposition (Ht-SVD) algorithm as follows. In MATLAB command, for the complex tensor $\mathcal{C} \in \mathbb{C}^{n_1\times n_2 \times n_3} $, the DFT of $\mathcal{C}$ and its inverse function are denoted as:
\begin{equation}
\hat{\mathcal{C}}=\text{fft}(\mathcal{C},\left[\,\right],3), \ \ \mathcal{C}=\text{ifft}(\hat{\mathcal{C}},\left[\,\right],3).\nonumber
\end{equation}

\begin{algorithm}
	\renewcommand{\algorithmicrequire}{\textbf{Input:}}
	\renewcommand{\algorithmicensure}{\textbf{Output:}}
	\caption{{\bf (SVD-RBT)} SVD for a Reduced  Biquaternion Tensor}
	\label{alg4}
	\begin{algorithmic}[1]
		\STATE $\textbf{Input}$: $\mathcal{A}_{(c),1},\mathcal{A}_{(c),2}\in \mathbb{C}^{n_1\times n_2 \times n_3}$ from given $\mathcal{A} =\mathcal{A}_{(c),1} e_1+ \mathcal{A}_{(c),2} e_2 \in \mathbb{H}_c^{n_1\times n_2 \times n_3}$ 
		\STATE  $\textbf{Output}$: $\mathcal{U}\in \mathbb{H}_c^{n_1 \times n_1  \times n_3}, \mathcal{S} \in \mathbb{H}_c^{n_1 \times n_2  \times n_3}$ and $\mathcal{V}\in \mathbb{H}_c^{n_2 \times n_2  \times n_3},$ such that
 \[
 \mathcal{A}=\mathcal{U}\ast_{Ht} \mathcal{S}\ast_{Ht} \mathcal{V}^\ast  
 \]  
           
\STATE   $ \hat{\mathcal{A}}_{(c),1}$=fft$( \mathcal{A}_{(c),1},[],3)$;$\hat{\mathcal{A}}_{(c),2}$=fft$( \mathcal{A}_{(c),2},[],3)$
\STATE  Calculate each frontal slice of $\hat{\mathcal{U}}, \hat{\mathcal{S}}$ and $\hat{\mathcal{V}}$ from $\hat{\mathcal{A}}$:\\
 $\textbf{for}$ $i=1,...n_{3},$ $\textbf{do}$\\
 $[\hat{U}_1^{(i)}, \hat{S}_1^{(i)}, \hat{V}_1^{(i)}]=svd(\hat{\mathcal{A}}_{(c),1}^{(i)})$,  $[\hat{U}_2^{(i)}, \hat{S}_2^{(i)}, \hat{V}_2^{(i)}]=svd(\hat{\mathcal{A}}_{(c),2}^{(i)})$\\
 $\textbf{end for}$
  \STATE $\mathcal{U}_1$=ifft($\hat{\mathcal{U}}_1$,$\left[\,\right]$,3), $\mathcal{S}_1$=ifft($\hat{\mathcal{S}}_1$,$\left[\,\right]$,3), $\mathcal{V}_1$=ifft($\hat{\mathcal{V}}_1$,$\left[\,\right]$,3); \\ $\mathcal{U}_2$=ifft($\hat{\mathcal{U}}_2$,$\left[\,\right]$,3), $\mathcal{S}_2$=ifft($\hat{\mathcal{S}}_2$,$\left[\,\right]$,3), $\mathcal{V}_2$=ifft($\hat{\mathcal{V}}$,$\left[\,\right]$,3)    
 \STATE $\mathcal{U} =\mathcal{U}_1 e_1+ \mathcal{U}_2 e_2;$
 $\mathcal{S} =\mathcal{S}_1 e_1+ \mathcal{S}_2 e_2;$
 $\mathcal{V} =\mathcal{V}_1 e_1+ \mathcal{V}_2 e_2;$
\STATE $\textbf{return}$   $\mathcal{U}$,
 $\mathcal{S}$,
 $\mathcal{V}$
	\end{algorithmic}  
\end{algorithm}

\noindent \textbf{Complexity Analysis of Algorithm 1: }
For a given reduced biquaternion tensor $\mathcal{A} \in \mathbb{H}_c^{n_1\times n_2 \times n_3}$, the computational complexity of each iteration of Algorithm 1 (\texttt{SVD-RBT}) is primarily driven by two major components: the fast fourier transform (FFT) and its inverse, and the SVD of the tensor's frontal 
slices. The computational complexity of  FFT and inverse FFT  operations is \(O(n_1 n_2 n_3 \log n_3)\). The SVD process, applied separately to each frontal slice, has a complexity of \(O(n_3 \min(n_1^2 n_2, n_1 n_2^2))\). Additionally, the tensor reconstruction step, which involves simple element-wise addition, has a complexity of \(O(n_1 n_2 n_3)\). Thus, the total computational complexity for each iteration of the algorithm is \(O(n_1 n_2 n_3 \log n_3 + n_3 \min(n_1^2 n_2, n_1 n_2^2))\). 

\section{The MP inverses of reduced biquaternion  tensors }
\label{MP}

We refer the readers to \cite{ben1974} for Moore-Penrose (MP) inverses.    In this section, as an application of Ht-SVD of a RB tensor, we define  the    MP inverse  for a RB tensor and find a formula for the MP inverse
by using Theorem \ref{HtSVD}. Moreover, we explore some theoretical properties of the MP inverses  of RB  tensors.

For a given matrix $A$, its MP inverse  is the unique   matrix $X$ (denoted as $A^\dagger$) satisfying
 \[
 AXA=A, \  XAX=X,  \  (AX)^\ast=AX,   \  (XA)^\ast=XA.
 \]
According to the properties of reduced biquaternion matrices, we can easily derive the MP inverse  of a reduced biquaternion matrix as follows.
\begin{theorem}
\label{defRBMP}
Let $A =A_{(c),1} e_1+ A_{(c),2} e_2 \in \mathbb{H}_c^{m\times n},$ where $A_{(c),1}, \  A_{(c),2} \in \mathbb{C}^{m\times n}$. The MP inverse of $A$ is   $A^\dagger =A_{(c),1}^\dagger e_1+ A_{(c),2}^\dagger e_2 \in \mathbb{H}_c^{n\times m}.$
\end{theorem}

Now, we introduce the definition of the MP inverse of a third-order reduced  biquaternion tensor.
\begin{definition}\label{defHtMP}
Let $ \mathcal{A} \in \mathbb{H}_c^{n_1\times n_2 \times n_3}$.
If there exists $ \mathcal{X} \in   \mathbb{H}_c^{n_2\times n_1 \times n_3} $ satisfying:
\begin{itemize}
  \item [(i)] $\mathcal{A} \ast_{Ht} \mathcal{X} \ast_{Ht} \mathcal{A}=\mathcal{A}; $
  \item [(ii)] $\mathcal{X} \ast_{Ht} \mathcal{A} \ast_{Ht} \mathcal{X}=\mathcal{X}; $
  \item [(iii)] $(\mathcal{A} \ast_{Ht} \mathcal{X})^\ast=\mathcal{A} \ast_{Ht} \mathcal{X}; $
  \item [(iv)] $(\mathcal{X} \ast_{Ht} \mathcal{A})^\ast=\mathcal{X} \ast_{Ht} \mathcal{A},$
\end{itemize}
then   $ \mathcal{X}$ is said to be an MP inverse of $\mathcal{A}$,   denoted   by $\mathcal{A}^\dagger.$
\end{definition}

For each $n_1\times n_2 \times n_3$ reduced biquaternion tensor, the following theorem shows that it has a unique MP  inverse.

\begin{theorem}
\label{HtMP}
 Let $ \mathcal{A} \in \mathbb{H}_c^{n_1\times n_2 \times n_3}$ and its  Ht-SVD   be  $  \mathcal{A}=\mathcal{U}\ast_{Ht} \mathcal{S}\ast_{Ht} \mathcal{V}^\ast, $
  where  $\mathcal{U}\in \mathbb{H}_c^{n_1 \times n_1  \times n_3}, \mathcal{V}\in \mathbb{H}_c^{n_2 \times n_2  \times n_3} $  are unitary  and
 $ \mathcal{S} \in \mathbb{H}_c^{n_1 \times n_2  \times n_3} $ is  f-diagonal.
Then $\mathcal{A}$ has a unique MP inverse
 \begin{equation}
 \label{HtMPfm}
 \mathcal{A}^\dagger = \mathcal{V}\ast_{Ht} \mathcal{S}^\dagger \ast_{Ht} \mathcal{U}^\ast,
 \end{equation}
where  $\mathcal{S}^\dagger = \textbf{Fold} \left(\frac{1}{\sqrt{n_3}} (F_{n_3}^\ast \otimes I_{n_2})    (\textbf{diag}(\widehat{\mathcal{S}}))^\dagger     (e \otimes I_{n_1}) \right)$.\end{theorem}

\begin{proof}
From Theorem \ref{Htthm}, Proposition \ref{Htassocia},  Proposition \ref{Htconjtr} and Definition \ref{defRBMP}, it is not difficult to verify that  $\textbf{diag}(\widehat{ \mathcal{S}^{\dagger} })= (\textbf{diag}(\widehat{\mathcal{S}}))^\dagger$  and the expression of $\mathcal{A}^\dagger$ given in (\ref{HtMPfm}) satisfies all four conditions in Definition \ref{defHtMP}.
 Next, we need to prove the uniqueness of the MP inverse.
Suppose that  $\mathcal{X}$ and $\mathcal{Y}$ are two MP inverses of     $\mathcal{A}$, respectively. Then
\begin{align*}
   \mathcal{X} & =\mathcal{X} \ast_{Ht} \mathcal{A} \ast_{Ht} \mathcal{X}=\mathcal{X} \ast_{Ht} \mathcal{X}^\ast \ast_{Ht} \mathcal{A}^\ast \\
   &= \mathcal{X} \ast_{Ht} \mathcal{X}^\ast \ast_{Ht} \mathcal{A}^\ast \ast_{Ht} \mathcal{Y}^\ast \ast_{Ht} \mathcal{A}^\ast\\
   &=\mathcal{X} \ast_{Ht} \mathcal{Y}^\ast \ast_{Ht} \mathcal{A}^\ast= \mathcal{X} \ast_{Ht} \mathcal{A} \ast_{Ht} \mathcal{Y} \\
   &=\mathcal{X} \ast_{Ht} \mathcal{A} \ast_{Ht} \mathcal{Y} \ast_{Ht} \mathcal{A} \ast_{Ht} \mathcal{Y}\\
   &= (\mathcal{X} \ast_{Ht} \mathcal{A})^\ast \ast_{Ht} (\mathcal{Y} \ast_{Ht} \mathcal{A})^\ast \ast_{Ht} \mathcal{Y}\\
   &=\mathcal{A}^\ast \ast_{Ht} \mathcal{Y}^\ast  \ast_{Ht} \mathcal{Y}= \mathcal{Y} \ast_{Ht} \mathcal{A} \ast_{Ht} \mathcal{Y}\\
   &=\mathcal{Y}.
\end{align*}
The proof is completed.
\end{proof}

Next, we explore some properties of  the MP inverse of a reduced biquaternion tensor   $\mathcal{A}$.
\begin{proposition}
\label{proRBMP}
Let $\mathcal{A} \in \mathbb{H}_c^{n_1\times n_2 \times n_3}$. Then
\begin{itemize}
  \item [(a)] $(\mathcal{A}^\dagger)^\dagger =\mathcal{A}$;
  \item [(b)] $(\mathcal{A}^\dagger)^\ast = (\mathcal{A}^\ast)^\dagger$;

  \item [(c)] $(\mathcal{A} \ast_{Ht} \mathcal{A}^\ast )^\dagger=(\mathcal{A}^\ast)^\dagger \ast_{Ht}  \mathcal{A} ^\dagger$  and $(\mathcal{A}^\ast \ast_{Ht} \mathcal{A})^\dagger= \mathcal{A}^\dagger \ast_{Ht}  (\mathcal{A}^\ast)^\dagger$;

  \item [(d)] $(\mathcal{P} \ast_{Ht} \mathcal{A} \ast_{Ht} \mathcal{Q})^\dagger = \mathcal{Q}^\ast  \ast_{Ht} \mathcal{A}^\dagger  \ast_{Ht} \mathcal{P}^\ast$ ( $\mathcal{P}$ and $\mathcal{Q}$ are unitary reduced biquaternion tensors );

  \item [(e)]  If $\mathcal{A}=( a_{i_1 i_2 i_3}) \in \mathbb{H}_c^{n_1\times n_1 \times n_3}$  is f-diagonal, then $\mathcal{A}^\dagger $  is f-diagonal,  and   $(\mathcal{A}^\dagger_{i_1 i_2 i_3})=( a_{i_1 i_2 i_3}) ^ \dagger $. In this case, if each diagonal entry in each frontal slice has an inverse, then $\mathcal{A}^{\dagger}=\mathcal{A}^{-1} $;

  \item [(f)] $\mathcal{A}^\dagger = \mathcal{A}^\ast  \ast_{Ht} (\mathcal{A} \ast_{Ht} \mathcal{A}^\ast)^\dagger = (\mathcal{A}^\ast \ast_{Ht} \mathcal{A})^\dagger \ast_{Ht}  \mathcal{A}^\ast$;

  \item [(g)] $\mathcal{A}=\mathcal{A} \ast_{Ht} \mathcal{A}^\ast \ast_{Ht} (\mathcal{A}^\ast)^\dagger =  (\mathcal{A}^\ast)^\dagger \ast_{Ht}  \mathcal{A}^\ast \ast_{Ht} \mathcal{A} $;

  \item [(h)] $\mathcal{A}^\ast=\mathcal{A}^\ast \ast_{Ht} \mathcal{A} \ast_{Ht} \mathcal{A}^\dagger = \mathcal{A}^\dagger \ast_{Ht} \mathcal{A} \ast_{Ht}  \mathcal{A}^\ast $;

  \item [(i)] $\mathcal{A} \ast_{Ht} \mathcal{A}^\dagger = (\mathcal{A}^\ast)^\dagger     \ast_{Ht} \mathcal{A}^\ast$ and  $\mathcal{A}^\dagger \ast_{Ht} \mathcal{A} = \mathcal{A}^\ast  \ast_{Ht}  (\mathcal{A}^\ast)^\dagger$;

  \item [(j)] $\mathcal{A} \ast_{Ht} \mathcal{A}^\dagger = (\mathcal{A} \ast_{Ht} \mathcal{A}^\ast )^\dagger  \ast_{Ht} \mathcal{A} \ast_{Ht} \mathcal{A}^\ast =\mathcal{A} \ast_{Ht} \mathcal{A}^\ast \ast_{Ht} (\mathcal{A} \ast_{Ht} \mathcal{A}^\ast )^\dagger $;

  \item [(k)] $\mathcal{A}^\dagger \ast_{Ht} \mathcal{A} = (\mathcal{A}^\ast \ast_{Ht} \mathcal{A} )^\dagger  \ast_{Ht} \mathcal{A}^\ast \ast_{Ht} \mathcal{A} =\mathcal{A}^\ast \ast_{Ht} \mathcal{A} \ast_{Ht} (\mathcal{A}^\ast \ast_{Ht} \mathcal{A})^\dagger $;

  \item [(l)] $\mathcal{A} \ast_{Ht} \mathcal{A}^\dagger=\mathcal{A}^\dagger \ast_{Ht} \mathcal{A} $  if $ \mathcal{A} \ast_{Ht} \mathcal{A}^\ast = \mathcal{A}^\ast \ast_{Ht} \mathcal{A} $;

  \item [(m)] $\mathcal{A} \ast_{Ht} \mathcal{A}^\dagger $ and $\mathcal{A}^\dagger \ast_{Ht} \mathcal{A}$  are idempotent;

  \item [(n)] If   $ \mathcal{A} $ is Hermitian and idempotent,  then  $\mathcal{A}^\dagger =  \mathcal{A}$.
\end{itemize}
\end{proposition}

\begin{proof}  We only prove (b), (c), (d), (l) and (n).  In the analogous way,  one can obtain the remaining properties.
\begin{itemize}
  \item [(b)]Using Propositions \ref{Htassocia}  and   \ref{Htconjtr},    one can obtain that $ (\mathcal{A}^\dagger)^\ast $ is the MP inverse of  $\mathcal{A}^\ast$ by direct calculations.

  \item [(c)]  By  (b), Proposition \ref{Htassocia},  Proposition \ref{Htconjtr} and the fact $(\mathcal{A}^\ast)^\ast=\mathcal{A}$,   one can get that $ (\mathcal{A}^\ast)^\dagger \ast_{Ht}  \mathcal{A}^\dagger$ is the MP inverse of  $\mathcal{A} \ast_{Ht} \mathcal{A}^\ast$  and   $ \mathcal{A}^\dagger \ast_{Ht} (\mathcal{A}^\ast)^\dagger   $ is the MP inverse of  $\mathcal{A}^\ast \ast_{Ht} \mathcal{A}$.
  \item [(d)] Suppose $\mathcal{P}  \ast_{Ht} \mathcal{P}^\ast =\mathcal{P}^\ast \ast_{Ht} \mathcal{P}= \mathcal{I} $ and $\mathcal{Q}  \ast_{Ht} \mathcal{Q}^\ast =\mathcal{Q}^\ast \ast_{Ht} \mathcal{Q}= \mathcal{I} $. Then  we  have
 \[
 \begin{array}{lll}
 & (\mathcal{P} \ast_{Ht} \mathcal{A} \ast_{Ht} \mathcal{Q})  \ast_{Ht} (\mathcal{Q}^\ast  \ast_{Ht} \mathcal{A}^\dagger  \ast_{Ht} \mathcal{P}^\ast)  \ast_{Ht}  (\mathcal{P} \ast_{Ht} \mathcal{A} \ast_{Ht} \mathcal{Q}) \\
 &  =\mathcal{P} \ast_{Ht} \mathcal{A} \ast_{Ht} \mathcal{A}^\dagger  \ast_{Ht} \mathcal{A} \ast_{Ht} \mathcal{Q}
  =\mathcal{P} \ast_{Ht} \mathcal{A}  \ast_{Ht} \mathcal{Q},
 \end{array}
 \]
 \[
 \begin{array}{lll}
 & (\mathcal{Q}^\ast  \ast_{Ht} \mathcal{A}^\dagger  \ast_{Ht} \mathcal{P}^\ast)  \ast_{Ht}  (\mathcal{P} \ast_{Ht} \mathcal{A} \ast_{Ht} \mathcal{Q}) \ast_{Ht} (\mathcal{Q}^\ast  \ast_{Ht} \mathcal{A}^\dagger  \ast_{Ht} \mathcal{P}^\ast) \\
 & =\mathcal{Q}^\ast  \ast_{Ht} \mathcal{A}^\dagger  \ast_{Ht} \mathcal{A}  \ast_{Ht}  \mathcal{A}^\dagger \ast_{Ht}  \mathcal{P}^\ast   =\mathcal{Q}^\ast  \ast_{Ht} \mathcal{A}^\dagger  \ast_{Ht}  \mathcal{P}^\ast,
  \end{array}
  \]
\[
 \begin{array}{lll}
&  (\mathcal{P} \ast_{Ht} \mathcal{A} \ast_{Ht} \mathcal{Q}  \ast_{Ht} \mathcal{Q}^\ast  \ast_{Ht} \mathcal{A}^\dagger  \ast_{Ht} \mathcal{P}^\ast)^\ast  \\
& =\mathcal{P} \ast_{Ht} \mathcal{A} \ast_{Ht} \mathcal{A}^\dagger  \ast_{Ht}  \mathcal{P}^\ast
    = \mathcal{P} \ast_{Ht} \mathcal{A} \ast_{Ht}   \mathcal{Q} \ast_{Ht}  \mathcal{Q}^\ast  \ast_{Ht} \mathcal{A}^\dagger  \ast_{Ht}  \mathcal{P}^\ast,
  \end{array}
  \]
       and
        $$(\mathcal{Q}^\ast  \ast_{Ht} \mathcal{A}^\dagger  \ast_{Ht} \mathcal{P}^\ast  \ast_{Ht}  \mathcal{P} \ast_{Ht} \mathcal{A} \ast_{Ht} \mathcal{Q})^\ast  =\mathcal{Q}^\ast  \ast_{Ht} \mathcal{A}^\dagger   \ast_{Ht} \mathcal{A} \ast_{Ht} \mathcal{Q} = \mathcal{Q}^\ast  \ast_{Ht} \mathcal{A}^\dagger  \ast_{Ht} \mathcal{P}^\ast  \ast_{Ht}  \mathcal{P} \ast_{Ht} \mathcal{A} \ast_{Ht} \mathcal{Q}.$$
        Hence,  by Definition   \ref{defHtMP} we can derive   that $(\mathcal{P} \ast_{Ht} \mathcal{A} \ast_{Ht} \mathcal{Q})^\dagger = \mathcal{Q}^\ast  \ast_{Ht} \mathcal{A}^\dagger  \ast_{Ht} \mathcal{P}^\ast.$
  \item [(l)] It follows from (i) and (j).
  \item [(n)] 
  If   $ \mathcal{A} $ is Hermitian and idempotent, then
$
\mathcal{A} \ast_{Ht} \mathcal{A} \ast_{Ht} \mathcal{A}  =  \mathcal{A} $  and $( \mathcal{A} \ast_{Ht} \mathcal{A}) ^\ast = \mathcal{A} ^\ast  \ast_{Ht} \mathcal{A}^\ast =\mathcal{A} \ast_{Ht} \mathcal{A}$.  Thus by Definition  \ref{defHtMP}  we have  $\mathcal{A}^\dagger =  \mathcal{A}$.

\end{itemize}
\end{proof}

In general,   $ (\mathcal{A} \ast_{Ht}  \mathcal{B} )^\dagger \neq \mathcal{B}^\dagger  \ast_{Ht}  \mathcal{A}^\dagger$. Under some conditions, we have the equality case given as follows.

\begin{proposition}
\label{proRBMPM}
Let $\mathcal{A} \in \mathbb{H}_c^{n_1\times n_2 \times n_3}$ and $\mathcal{B} \in \mathbb{H}_c^{n_2\times n_4 \times n_3}$. Then $ (\mathcal{A} \ast_{Ht}  \mathcal{B} )^\dagger= \mathcal{B}^\dagger  \ast_{Ht}  \mathcal{A}^\dagger$ if one of the following conditions  satisfies.
\begin{center}
 (a) $ \mathcal{A}= \mathcal{B} ^\dagger$; \  (b) $\mathcal{A}= \mathcal{B} ^\ast$; \
(c) $\mathcal{A}^\ast  \ast_{Ht}  \mathcal{A}=\mathcal{I}$; \
 (d) $\mathcal{B} \ast_{Ht} \mathcal{B}^\ast=\mathcal{I}$.
\end{center}
\end{proposition}

Denote $\mathcal{L}_{\mathcal{A}}=\mathcal{I}-{\mathcal{A}}^\dagger  \ast_{Ht} \mathcal{A}, \ \mathcal{R}_{\mathcal{A}}=\mathcal{I}-\mathcal{A}  \ast_{Ht} {\mathcal{A}}^\dagger.$
$\mathcal{O}$ represents  zero tensor with  all the entries being zero.
 Obviously, $ \mathcal{A}  \ast_{Ht}  \mathcal{O} = \mathcal{O} \ast_{Ht} \mathcal{A} =\mathcal{O}.$  

\begin{proposition}
\label{proLRRBMP}
Let $\mathcal{A} \in \mathbb{H}_c^{n_1\times n_2 \times n_3}$. Then
\begin{itemize}
  \item [(a)] $ \mathcal{A} \ast_{Ht}  \mathcal{L}_{\mathcal{A}} =\mathcal{O}$ and $\mathcal{R}_{\mathcal{A}} \ast_{Ht}  \mathcal{A} =\mathcal{O}$;
  \item [(b)] $\mathcal{L}_{\mathcal{A}} \ast_{Ht}  \mathcal{A}^\dagger =\mathcal{O}$ and $ \mathcal{A}^\dagger \ast_{Ht}  \mathcal{R}_{\mathcal{A}} =\mathcal{O}$;
  \item [(c)] $\mathcal{L}_{\mathcal{A}^\ast} =(\mathcal{R}_{\mathcal{A}})^\ast= \mathcal{R}_{\mathcal{A}}$ and $\mathcal{R}_{\mathcal{A}^ \ast} =(\mathcal{L}_{\mathcal{A}})^\ast = \mathcal{L}_{\mathcal{A}}$;
  \item [(d)] $\mathcal{L}_{\mathcal{A}} \ast_{Ht} \mathcal{L}_{\mathcal{A}}=\mathcal{L}_{\mathcal{A}}$ and $\mathcal{R}_{\mathcal{A}} \ast_{Ht} \mathcal{R}_{\mathcal{A}}=\mathcal{R}_{\mathcal{A}}$;
  \item [(e)] $(\mathcal{L}_{\mathcal{A}})^\dagger = \mathcal{L}_{\mathcal{A}}$ and $(\mathcal{R}_{\mathcal{A}})^\dagger = \mathcal{R}_{\mathcal{A}}$;
  \item [(f)] $\mathcal{L}_{\mathcal{A}} =  \mathcal{L}_{ \mathcal{A}^\ast  \ast_{Ht} {\mathcal{A}}}$ and $\mathcal{R}_{\mathcal{A}} =  \mathcal{R}_{ \mathcal{A}  \ast_{Ht} {\mathcal{A}^\ast}}$.
\end{itemize}
\end{proposition}

\begin{proof} By straightforward  calculations, one can prove that (a), (b) and (d) hold.

For (c),  the first equality follows from
$$\mathcal{L}_{\mathcal{A}^\ast} = \mathcal{I}-(\mathcal{A}^\ast)^\dagger  \ast_{Ht}  {\mathcal{A}^ \ast}=\mathcal{I}- (\mathcal{A} \ast_{Ht}  \mathcal{A}^\dagger)^\ast
=(\mathcal{I}- \mathcal{A} \ast_{Ht}  \mathcal{A}^\dagger)^\ast =  (\mathcal{R}_{\mathcal{A}})^\ast$$ and
$$\mathcal{L}_{\mathcal{A}^\ast} = \mathcal{I}-(\mathcal{A}^\ast)^\dagger  \ast_{Ht}  {\mathcal{A}^ \ast}=\mathcal{I}- (\mathcal{A} \ast_{Ht}  \mathcal{A}^\dagger)^\ast
=\mathcal{I}- \mathcal{A} \ast_{Ht}  \mathcal{A}^\dagger =  \mathcal{R}_{\mathcal{A}}.$$
In an analogous manner, one can prove the second part.

For (e), it follows immediately from (c), (d) and Proposition \ref{proRBMP}.

For (f), 
$ \mathcal{L}_{ \mathcal{A}^\ast  \ast_{Ht} {\mathcal{A}}} = \mathcal{I}-( \mathcal{A}^\ast  \ast_{Ht} {\mathcal{A}} )^\dagger  \ast_{Ht} ( \mathcal{A}^\ast  \ast_{Ht} {\mathcal{A}}) = \mathcal{I}-   {\mathcal{A}} ^\dagger  \ast_{Ht}  ({\mathcal{A}} ^\dagger)^\ast     \ast_{Ht}  \mathcal{A}^\ast  \ast_{Ht} {\mathcal{A}} = \mathcal{I}-   {\mathcal{A}} ^\dagger  \ast_{Ht}  {\mathcal{A}}     \ast_{Ht}  \mathcal{A}^\dagger  \ast_{Ht} {\mathcal{A}} = \mathcal{I}-   {\mathcal{A}} ^\dagger  \ast_{Ht}  {\mathcal{A}} =\mathcal{L}_{\mathcal{A}}.$  The second part  can be proved in a similar way.
\end{proof}

\section{  Reduced biquaternion tensor equation $\mathcal{A} \ast_{Ht} \mathcal{X}=\mathcal{B}$ }

We mainly discuss the consistency of    RB tensor equation
\begin{equation}\label{EQAXB}
  \mathcal{A} \ast_{Ht} \mathcal{X}=\mathcal{B}
\end{equation}
by using the MP inverse demonstrated in Section \ref{MP}.
   The solvability conditions  are found and    a general expression of the solutions is provided  when the solutions exist. Except that, we derive the least-squares solutions/the minimal norm least-squares solution to \eqref{EQAXB}. Furthermore, we also derive solvability conditions and the general solutions involving Hermicity.
One of the applications of solving this tensor equations is color video deblurring given in Section 6.

\begin{theorem}
\label{HtEQ}
Let $\mathcal{A} \in \mathbb{H}_c^{n_1\times n_2 \times n_3}$ and $\mathcal{B} \in \mathbb{H}_c^{n_1\times n_4 \times n_3}.$
Then     equation (\ref{EQAXB})
 has solutions if and only if   $\mathcal{R}_{\mathcal{A}} \ast_{Ht} \mathcal{B}= \mathcal{O} $. Furthermore, the general  solutions to   equation (\ref{EQAXB})
can be written as
\begin{equation}\label{EQ}
\mathcal{X} =\mathcal{A} ^\dagger \ast_{Ht}  \mathcal{B} + \mathcal{L}_{\mathcal{A}} \ast_{Ht} \mathcal{Y},
\end{equation}
with $\mathcal{Y} $ be any reduced biquaternion  tensor with compatible size.
\end{theorem}

\begin{proof}
Assume that  $\mathcal{C}$ is a solution to  reduced biquaternion tensor    equation (\ref{EQAXB}).
Then by Proposition \ref{Htassocia} and Proposition \ref{proLRRBMP},  we have
$$
 \mathcal{R}_{\mathcal{A}} \ast_{Ht} \mathcal{B}= \mathcal{R}_{\mathcal{A}}  \ast_{Ht} \mathcal{A} \ast_{Ht} \mathcal{C}
=\mathcal{O}.
$$

Conversely, if $\mathcal{R}_{\mathcal{A}} \ast_{Ht} \mathcal{B}= \mathcal{O}$,
it is easy to see that
$ \mathcal{X}_0=\mathcal{A}^\dagger \ast_{Ht} \mathcal{B}$  is a solution to     equation (\ref{EQAXB}), 
and so is   $\mathcal{A}^\dagger \ast_{Ht}  \mathcal{B} + \mathcal{L}_{\mathcal{A}} \ast_{Ht} \mathcal{Y}$.

Next we show each solution $\mathcal{X}$  of     equation (\ref{EQAXB})
has the form of $\mathcal{X} =\mathcal{A} ^\dagger \ast_{Ht}  \mathcal{B} +\mathcal{L}_{\mathcal{A}} \ast_{Ht} \mathcal{X}$. Note that $ \mathcal{A} \ast_{Ht} \mathcal{X}=\mathcal{B}$,
thus $\mathcal{X} = \mathcal{A}^\dagger \ast_{Ht}  \mathcal{B} -\mathcal{A}^\dagger \ast_{Ht} \mathcal{A} \ast_{Ht} \mathcal{X}+\mathcal{X}= \mathcal{A}^\dagger \ast_{Ht}  \mathcal{B}+(\mathcal{I}-{\mathcal{A}}^\dagger  \ast_{Ht} \mathcal{A}) \ast_{Ht} \mathcal{X},$ as required.
Based on the above, we can conclude that \eqref{EQ} is the general expression of the solutions.
 \end{proof}

\begin{theorem}
\label{HtEQH}
Let $\mathcal{A},  \mathcal{B} \in \mathbb{H}_c^{n_1\times n_2 \times n_3}$.  Then equation (\ref{EQAXB})
has a Hermitian solution if and only if
$\mathcal{A} \ast_{Ht} \mathcal{B}^\ast=\mathcal{B} \ast_{Ht} \mathcal{A}^\ast $
 and
 $\mathcal{R}_{\mathcal{A}} \ast_{Ht}  \mathcal{B} = \mathcal{O}.$ Moreover, if (\ref{EQAXB}) is solvable, then the general Hermitian solutions to the reduced  biquaternion tensor equation (\ref{EQAXB})
can be written as
\begin{equation}\label{EQH}
\mathcal{X} =\mathcal{A}^\dagger  \ast_{Ht} \mathcal{B}+(\mathcal{A}^\dagger  \ast_{Ht} \mathcal{B})^\ast-\mathcal{A}^\dagger  \ast_{Ht} ( \mathcal{A}\ast_{Ht}  \mathcal{B}^\ast) \ast_{Ht}  (\mathcal{A}^\dagger)^\ast + \mathcal{L}_{\mathcal{A}} \ast_{Ht} \mathcal{U} \ast_{Ht}  \mathcal{L}_{\mathcal{A}},
\end{equation}
where $\mathcal{U}=\mathcal{U}^\ast $  is an arbitrary   tensor with appropriate size.
\end{theorem}

\begin{proof}
If     equation (\ref{EQAXB})
has a  Hermitian solution $\mathcal{C} $,
then  by  Propositions \ref{Htassocia},  \ref{Htconjtr} and  Definition \ref{defHtMP},  we obtain
\[
\mathcal{A} \ast_{Ht} \mathcal{B}^\ast= \mathcal{A} \ast_{Ht} \mathcal{C}^\ast \ast_{Ht} \mathcal{A}^\ast = \mathcal{A} \ast_{Ht} \mathcal{C}  \ast_{Ht} \mathcal{A}^\ast = \mathcal{B} \ast_{Ht} \mathcal{A}^\ast,
\]
 and  $\mathcal{R}_{\mathcal{A}} \ast_{Ht} \mathcal{B} =\mathcal{R}_{\mathcal{A}} \ast_{Ht}  \mathcal{A} \ast_{Ht} \mathcal{C} =  \mathcal{O} .$

On the other hand, assume that $\mathcal{A} \ast_{Ht} \mathcal{B}^\ast=\mathcal{B} \ast_{Ht} \mathcal{A}^\ast $
 and $R_{\mathcal{A}} \ast_{Ht}  \mathcal{B} = \mathcal{O} $ hold.
 Then by applying Propositions \ref{Htassocia} and   \ref{Htconjtr}, we get that
 \[
 \mathcal{A}^\dagger  \ast_{Ht} \mathcal{B}+(\mathcal{A}^\dagger  \ast_{Ht} \mathcal{B})^\ast-\mathcal{A}^\dagger  \ast_{Ht} ( \mathcal{A}\ast_{Ht}  \mathcal{B}^\ast) \ast_{Ht}  (\mathcal{A}^\dagger)^\ast + \mathcal{L}_{\mathcal{A}}  \ast_{Ht} \mathcal{U} \ast_{Ht}   \mathcal{L}_{\mathcal{A}}
 \]
 is a Hermitian solution to   reduced  biquaternion  tensor equation (\ref{EQAXB}).

Next we show that each solution of   reduced  biquaternion  tensor equation (\ref{EQAXB})
should be in the form of  (\ref{EQH}).
Assume $\mathcal{X}_0$ is an arbitrary Hermitian solution to    tensor equation (\ref{EQAXB}). Setting  $\mathcal{U}=\mathcal{X}_0$  yields
\begin{align*}
  \mathcal{X}   = & \mathcal{A}^\dagger  \ast_{Ht} \mathcal{B}+(\mathcal{A}^\dagger  \ast_{Ht} \mathcal{B})^\ast-\mathcal{A}^\dagger  \ast_{Ht} ( \mathcal{A}\ast_{Ht}  \mathcal{B}^\ast) \ast_{Ht}  (\mathcal{A}^\dagger)^\ast + \mathcal{L}_{\mathcal{A}} \ast_{Ht} \mathcal{X}_0 \ast_{Ht}   \mathcal{L}_{\mathcal{A}}\\
    = & \mathcal{A}^\dagger  \ast_{Ht} \mathcal{B}+(\mathcal{A}^\dagger  \ast_{Ht} \mathcal{B})^\ast-\mathcal{A}^\dagger  \ast_{Ht} ( \mathcal{A}\ast_{Ht}  \mathcal{B}^\ast) \ast_{Ht}  (\mathcal{A}^\dagger)^\ast  \\
    & + ( \mathcal{X}_0 -\mathcal{A}^\dagger  \ast_{Ht}  \mathcal{A} \ast_{Ht}  \mathcal{X}_0 ) \ast_{Ht} (\mathcal{I}- \mathcal{A}^\dagger  \ast_{Ht}  \mathcal{A})\\
    = & (\mathcal{A}^\dagger  \ast_{Ht} \mathcal{B})^\ast-\mathcal{A}^\dagger  \ast_{Ht} ( \mathcal{A}\ast_{Ht}  \mathcal{B}^\ast) \ast_{Ht}  (\mathcal{A}^\dagger)^\ast  \\
    & + \mathcal{X}_0 -\mathcal{X}_0 \ast_{Ht} (\mathcal{A}^\dagger  \ast_{Ht}  \mathcal{A})+ \mathcal{A}^\dagger  \ast_{Ht}  \mathcal{B}   \ast_{Ht}   \mathcal{A}^\ast \ast_{Ht}  (\mathcal{A}^\dagger)^\ast\\
   = & (\mathcal{A}^\dagger  \ast_{Ht} \mathcal{A} \ast_{Ht}  \mathcal{X}_0 )^\ast+ \mathcal{X}_0 -\mathcal{X}_0 \ast_{Ht} (\mathcal{A}^\dagger  \ast_{Ht}  \mathcal{A})^\ast\\
   = & \mathcal{X}_0,
\end{align*}
which implies that each Hermitian solution to   reduced  biquaternion  tensor equation (\ref{EQAXB})
can be represented as (\ref{EQH}). Thus,  (\ref{EQH}) is the general expression of Hermitian
solutions to    equation (\ref{EQAXB}).
\end{proof}

Next, we deal with the least-squares problem of  RB tensor equation (\ref{EQAXB}).
For a        tensor   $\mathcal{A}=\mathcal{A}_{1,\mathbf{i}}+\mathbf{j} \mathcal{A}_{\mathbf{j},\mathbf{k}} \in  \mathbb{H}_c^{n_1\times n_2 \times n_3}$
with $\mathcal{A}_{1,\mathbf{i}}, \mathcal{A}_{\mathbf{j},\mathbf{k}} \in \mathbb{C}^{n_1\times n_2 \times n_3} $, we define its Frobenius norm   as
\begin{equation}
\label{frob}
 \|  \mathcal{A} \|_F^2= \| \textbf{Vec}(\mathcal{A}) \|_F ^2= \| \textbf{Vec}(\mathcal{A}_{1,\mathbf{i}}) \|_F ^2+\| \textbf{Vec}(\mathcal{A}_{\mathbf{j},\mathbf{k}}) \|_F ^2.
\end{equation}
where 
 $\textbf{Vec}(\mathcal{A}_{1,\mathbf{i}}) =(c_{ij})_{n_1n_3\times {n_2}},$
 $\textbf{Vec}(\mathcal{A}_{\mathbf{j},\mathbf{k}}) =(d_{ij})_{n_1n_3\times {n_2}}.$

 \begin{theorem}
\label{HtF}
Let $\mathcal{A} \in \mathbb{H}_c^{n_1\times n_2 \times n_3}.$   Then $  \|   \mathcal{A} \|_F =\frac{1 }{\sqrt{n_3}}  \| \textbf{diag}(\hat{\mathcal{A}}) \|_F$.
 \end{theorem}

\begin{proof}
 By applying identities (\ref{hatA}), (\ref{frob})  and  $F_{n_3} \mathbf{j}=\mathbf{j}F_{n_3}$  as well as the invariant property of Frobenius norm under a unitary transformation for
  complex matrices, we have
\begin{align*}
   \| \textbf{diag}(\hat{\mathcal{A}}) \|_F^2 &=\| \textbf{Vec}(\hat{\mathcal{A}}) \|_F^2 \\
   & =  \| \sqrt{n_3} (F_{n_3} \otimes I_{n_1})  \textbf{Vec}(\mathcal{A}_{1,\mathbf{i}}) +     \sqrt{n_3} (F_{n_3} \otimes I_{n_1})  \mathbf{j} \textbf{Vec}(\mathcal{A}_{\mathbf{j},\mathbf{k}})   \|_F^2 \\
   &=\| \sqrt{n_3} (F_{n_3} \otimes I_{n_1})  \textbf{Vec}(\mathcal{A}_{1,\mathbf{i}}) \|_F^2 +     \|\sqrt{n_3} (F_{n_3} \otimes I_{n_1})   \textbf{Vec}(\mathcal{A}_{\mathbf{j},\mathbf{k}})   \|_F^2 \\
   & = n_3  \|   \textbf{Vec}(\mathcal{A}_{1,\mathbf{i}})\|_F^2   +
       n_3  \| \textbf{Vec}(\mathcal{A}_{\mathbf{j},\mathbf{k}}) \|_F^2   \\
   & = n_3  \|   \textbf{Vec}(\mathcal{A}) \|_F^2 \\
   & = n_3  \|   \mathcal{A} \|_F^2.
\end{align*}
Therefore, $\|   \mathcal{A} \|_F =\frac{1 }{\sqrt{n_3}}  \| \textbf{diag}(\hat{\mathcal{A}}) \|_F $.
\end{proof}

\begin{theorem}
\label{HtLSMP}
Let $\mathcal{A} \in \mathbb{H}_c^{n_1\times n_2 \times n_3}$ and $\mathcal{B} \in \mathbb{H}_c^{n_1\times n_4 \times n_3}.$
Then the least-squares solutions of tensor equation (\ref{EQAXB}) can be  represented as   $\mathcal{X} =\mathcal{A}^\dagger \ast_{Ht}  \mathcal{B}  + (\mathcal{I}- \mathcal{A}^\dagger \ast_{Ht}\mathcal{A})\ast_{Ht} \mathcal{W} $,  where $\mathcal{W} $ is an arbitrary compatible RB tensor,
and the minimum-norm least-squares solution of tensor equation (\ref{EQAXB}) is
$\mathcal{X} =\mathcal{A}^\dagger \ast_{Ht}  \mathcal{B} $.  
\end{theorem}
\begin{proof}
According to  Theorem \ref{Htthm} and Theorem \ref{HtF}, we have
$$\| \mathcal{A} \ast_{Ht} \mathcal{X}-\mathcal{B} \|_F  =\frac{1 }{\sqrt{n_3}}  \| \textbf{diag}(\hat{\mathcal{A}})  \textbf{diag}(\hat{\mathcal{X}})-\textbf{diag}(\hat{\mathcal{B}}) \|_F. $$
Thus, they arrive the minima at the same time. We know that the the least-squares solutions of RB matrix equation $ \textbf{diag}(\hat{\mathcal{A}})  \textbf{diag}(\hat{\mathcal{X}})=\textbf{diag}(\hat{\mathcal{B}})$ are  
$$ \textbf{diag}(\hat{\mathcal{X}}) = \textbf{diag}(\hat{\mathcal{A}})^\dagger \textbf{diag}(\hat{\mathcal{B}})+[\textbf{diag}(\hat{\mathcal{I}})-\textbf{diag}(\hat{\mathcal{A}})^\dagger \textbf{diag}(\hat{\mathcal{A}})] \textbf{diag}(\hat{\mathcal{W}})$$
with $\textbf{diag}(\hat{\mathcal{W}})$ be arbitrary.
By using the fact that $\textbf{diag}(\hat{\mathcal{A}})^\dagger =\textbf{diag}(\hat{\mathcal{A}^\dagger}) $ and  Theorem \ref{Htthm} we can derive 
$$\mathcal{X}=  \mathcal{A}^\dagger  \ast_{Ht} \mathcal{B}+(\mathcal{I}-\mathcal{A}^\dagger \ast_{Ht} \mathcal{A})  \ast_{Ht} \mathcal{W}$$ is  the least-squares solution of  $\mathcal{A} \ast_{Ht} \mathcal{X}=\mathcal{B}$. For any least-squares solution $ \mathcal{X}_0$, we have
\begin{align*}
\| \mathcal{X}_0\|^2_F  =&\|   \mathcal{A}^\dagger \ast_{Ht} \mathcal{B}+(\mathcal{I}-\mathcal{A}^\dagger \ast_{Ht} \mathcal{A}) \ast_{Ht} \mathcal{W}\|^2_F  \\
=& \frac{1 }{n_3} \|   \textbf{diag}(\hat{\mathcal{A}})^\dagger \textbf{diag}(\hat{\mathcal{B}})+(\textbf{diag}(\hat{\mathcal{I}})-\textbf{diag}(\hat{\mathcal{A}})^\dagger \textbf{diag}(\hat{\mathcal{A}})) \textbf{diag}(\hat{\mathcal{W}}) \|^2_F   \\
= &\frac{1 }{n_3}  [\|   \textbf{diag}(\hat{\mathcal{A}})^\dagger \textbf{diag}(\hat{\mathcal{B}})\|^2_F+
\|  \textbf{diag}(\hat{\mathcal{I}})-\textbf{diag}(\hat{\mathcal{A}})^\dagger \textbf{diag}(\hat{\mathcal{A}}) ) \textbf{diag}(\hat{\mathcal{W}}) \|^2_F \\
& +  2tr(\Re ( (\textbf{diag}(\hat{\mathcal{A}})^\dagger \textbf{diag}(\hat{\mathcal{B}}))^\ast (\textbf{diag}(\hat{\mathcal{I}})-\textbf{diag}(\hat{\mathcal{A}})^\dagger \textbf{diag}(\hat{\mathcal{A}}))  \textbf{diag}(\hat{\mathcal{W}}) ) ) ].
\end{align*}
Since 
\begin{align*}
& (\textbf{diag}(\hat{\mathcal{A}})^\dagger \textbf{diag}(\hat{\mathcal{B}}))^\ast (\textbf{diag}(\hat{\mathcal{I}})-\textbf{diag}(\hat{\mathcal{A}})^\dagger \textbf{diag}(\hat{\mathcal{A}}))  \textbf{diag}(\hat{\mathcal{W}}) \\
=&  \textbf{diag}(\hat{\mathcal{B}}))^\ast (\textbf{diag}(\hat{\mathcal{A}})^\dagger)^\ast( \textbf{diag}(\hat{\mathcal{I}})-(\textbf{diag}(\hat{\mathcal{A}})^\dagger \textbf{diag}(\hat{\mathcal{A}}))^\ast )  \textbf{diag}(\hat{\mathcal{W}})  \\
=& \textbf{diag}(\hat{\mathcal{B}}))^\ast (  (\textbf{diag}(\hat{\mathcal{A}})^\dagger)^\ast-(\textbf{diag}(\hat{\mathcal{A}})^\dagger \textbf{diag}(\hat{\mathcal{A}}) \textbf{diag}(\hat{\mathcal{A}})^\dagger) ^\ast )  \textbf{diag}(\hat{\mathcal{W}}) \\
=& O,
\end{align*}
we can conclude that 
$$\| \mathcal{X}_0\|^2_F = \|\mathcal{A}^\dagger   \ast_{Ht} \mathcal{B}  \|^2_F + \|(\mathcal{I}- \mathcal{A}^\dagger \ast_{Ht} \mathcal{A} )  \ast_{Ht}\mathcal{W}  \|^2_F  \geq   \|\mathcal{A}^\dagger   \ast_{Ht} \mathcal{B}  \|^2_F.     
$$
Consequently,  the  least-squares solution of  tensor equation (\ref{EQAXB})  with minimum norm is 
$\mathcal{X} =\mathcal{A}^\dagger \ast_{Ht}  \mathcal{B} $.
\end{proof}

In an analogous way, we have the following result.  For simplicity, we omit the proof here.
\begin{theorem}
\label{HtLSMP2}
Let $\mathcal{A} \in \mathbb{H}_c^{n_1\times n_2 \times n_3}$ and $\mathcal{B} \in \mathbb{H}_c^{n_4\times n_2 \times n_3}.$
Then the least-squares solutions of RB tensor equation $\mathcal{X} \ast_{Ht} \mathcal{A} =\mathcal{B}$  can be  represented as   $\mathcal{X} =\mathcal{B} \ast_{Ht} \mathcal{A}^\dagger +\mathcal{W} \ast_{Ht}  (\mathcal{I}- \mathcal{A} \ast_{Ht}\mathcal{A}^\dagger) $,  where $\mathcal{W} $ is an arbitrary compatible RB tensor,  
and the the minimum-norm least-squares solutions of RB tensor equation $ \mathcal{X} \ast_{Ht} \mathcal{A} =\mathcal{B} $   is
$\mathcal{X} =\mathcal{B} \ast_{Ht}  \mathcal{A}^\dagger $.  
\end{theorem}

 \section{Applications in color video processing}
 In the section, we propose two algorithms for the color video compression and debluring. Both experiments of color video processing show the  superiority of our methods over the compared methods.

Based on the Ht-SVD, we  define Ht-tubal rank for  the tensor $\mathcal{A}\in \mathbb{H}_c^{n_{1} \times n_{2}\times n_{3}}$ and the rank-$k$  approximation $\mathcal{A}_k$ of $\mathcal{A}$ as follows.
\begin{definition}\label{rank}
 For a   tensor $\mathcal{A}\in \mathbb{H}_c^{n_{1} \times n_{2}\times n_{3}}$ with Ht-SVD, i.e,  $\mathcal{A} = \mathcal{U}*_{Ht}\mathcal{S}*_{Ht}\mathcal{V}^*$. The Ht-tubal rank  $rank_{Ht-tubal}(\mathcal{A})$ of $\mathcal{A}$ is defined to be the number of nonzero elements of ${\mathcal{S}(i,i,:)}^{w}_{i=1}$, where $w \doteq \text{min}(n_{1},n_{2})$, i.e.,
 \begin{equation}
rank_{Ht-tubal}(\mathcal{A})= \#\{i\mid \Vert\mathcal{S}(i,i,:)\Vert_{F}>0\}.\nonumber   
 \end{equation}
\end{definition}   
When we compute the SVD of frontal slices $\hat{\mathcal{A}}^{(1)}, \cdots, \hat{\mathcal{A}}^{(n_3)}$ with singular values { in decreasing order},  $rank_{Ht-tubal}(\mathcal{A})=\mathop{\max}\limits_{i}\{rank(\hat{\mathcal{A}}^{(1)}), \cdots, rank(\hat{\mathcal{A}}^{(n_3)})\}$.
If we assume $k \leq rank_{Ht-tubal}(\mathcal{A})$, then 
\begin{equation}
\label{Ak}
\mathcal{A}_k=\sum\limits_{i=1}^{k} \mathcal{U}(:,i,:)\ast_{Ht} \mathcal{S}(i, i,:)\ast_{Ht} \mathcal{V}(:,i,:)^\ast 
\end{equation}
has the Ht-tubal rank of $k.$ Next, we will apply the  approximation $\mathcal{A}_k$ of $\mathcal{A}$ to compress the color video.
\\
\noindent \textbf{Example 6.2: (Color video compression)} We use a reduced biquaternion tensor $\mathcal{A} \in \mathbb{H}_c^{n_1\times n_2 \times n_3}$ to express a color video,  i.e.,
\[\mathcal{A}_{C}(m,n,t)=\mathcal{A}_{R}(m,n,t)\mathbf{i}+\mathcal{A}_{G}(m,n,t)\mathbf{j}+\mathcal{A}_{B}(m,n,t)\mathbf{k},\]
where  the red $\mathcal{A}_{R}(m,n,t)$, green $\mathcal{A}_{G}(m,n,t)$ and blue  $\mathcal{A}_{B}(m,n,t)$  pixel values are sets    of the $t$-th color image of a color video. 

\begin{algorithm}
	\renewcommand{\algorithmicrequire}{\textbf{Input:}}
	\renewcommand{\algorithmicensure}{\textbf{Output:}}
	\caption{{\bf (RkA-RBT)} Rank-$k$  Approximation $\mathcal{A}_k$ of Reduced  Biquaternion Tensor $\mathcal{A}$}
	\label{alg5}
	\begin{algorithmic}[1]
		\STATE $\textbf{Input}$: $\mathcal{A}_{(c),1},\mathcal{A}_{(c),2}\in \mathbb{C}^{n_1\times n_2 \times n_3}$ from given $\mathcal{A} =\mathcal{A}_{(c),1} e_1+ \mathcal{A}_{(c),2} e_2 \in \mathbb{H}_c^{n_1\times n_2 \times n_3}$, given $k$
		\STATE  $\textbf{Output}$: $\mathcal{A}_k\in \mathbb{H}_c^{n_1 \times n_2\textbf{}  \times n_3}$ such that
 \[
 \mathcal{A}_k=\mathcal{U}_k\ast_{Ht} \mathcal{S}_k\ast_{Ht} \mathcal{V}_k^\ast  
 \]  
           
\STATE   $ \hat{\mathcal{A}}_{(c),1}$=fft$( \mathcal{A}_{(c),1},[],3)$;$\hat{\mathcal{A}}_{(c),2}$=fft$( \mathcal{A}_{(c),2},[],3)$
\STATE  Calculate each frontal slice of $\hat{\mathcal{U}}, \hat{\mathcal{S}}$ and $\hat{\mathcal{V}}$ from $\hat{\mathcal{A}}$:\\
 $\textbf{for}$ $i=1,...n_{3},$ $\textbf{do}$\\
 $[\hat{U}_1^{(i)}, \hat{S}_1^{(i)}, \hat{V}_1^{(i)}]=svd(\hat{\mathcal{A}}_{(c),1}^{(i)})$,  $[\hat{U}_2^{(i)}, \hat{S}_2^{(i)}, \hat{V}_2^{(i)}]=svd(\hat{\mathcal{A}}_{(c),2}^{(i)})$\\
 $\textbf{end for}$
  \STATE Calculate the approximate frontal slice $\hat{\mathcal{A}}_{(c),1}^{(i)}(k), \hat{\mathcal{A}}_{(c),2}^{(i)}(k)$ of $\hat{\mathcal{A}}_{(c),1}(k), \hat{\mathcal{A}}_{(c),2}(k)$  :\\
$\textbf{for}$ $i=1,...n_{3},$ $\textbf{do}$\\
 $\hat{\mathcal{A}}_{(c),1}^{(i)}(k)=\hat{U}_1^{(i)}(:,1:k)\hat{S}_1^{(i)}(1:k,1:k)\hat{V}_1^{(i)}(:,1:k)^\ast$\\
 $\hat{\mathcal{A}}_{(c),2}^{(i)}(k)=\hat{U}_2^{(i)}(:,1:k)\hat{S}_2^{(i)}(1:k,1:k)\hat{V}_2^{(i)}(:,1:k)^\ast$\\
  $\textbf{end for}$
  \STATE $\mathcal{A}_{(c),1}(k)$=ifft($\hat{\mathcal{A}}_{(c),1}(k)$,$\left[\,\right]$,3) \\
  $\mathcal{A}_{(c),2}(k)$=ifft($\hat{\mathcal{A}}_{(c),2}(k)$,$\left[\,\right]$,3)
 \STATE $\mathcal{A}_k=\mathcal{A}_{(c),1}(k) e_1+ \mathcal{A}_{(c),2}(k) e_2$
\STATE $\textbf{return}$   $\mathcal{A}_k$.
	\end{algorithmic}  
\end{algorithm}
For comparison, we apply \texttt{SVD-RBT} to compress the same color video ``DO01\_013" as the one in \cite{qinzhanglp2022}, while in \cite{qinzhanglp2022}, it was compressed by \texttt{Qt-SVD}. Fig. \ref{psnr} shows the PSNRs of the compressed video when taking the rank-$k$ $(k=10, 20, 50)$ approximations $\mathcal{A}_k$ in (\ref{Ak})
to ``DO01\_013" for the 1st, the 20th and the 50th frames, respectively. 

To compute  the PSNR of each frame, we can treat each frame as a real tensor $\mathcal{C}\in \mathbb{R}^{n_1 \times n_2 \times 3}$, and the approximate frame as $\mathcal{C}_k\in \mathbb{R}^{n_1 \times n_2 \times 3}$, thus its PSNR is defined as
\begin{equation}\label{psnr2}
PSNR \doteq 10log_{10}\left(\frac{3n_1n_2\|\mathcal{C}\|_{\infty}^{2}}{\|\mathcal{C}-\mathcal{C}_k\|_{F}^{2}}\right),
\end{equation}
with $\mathcal{C}=(C_{mnt})_{n_1 \times n_2 \times 3}, \mathcal{C}_k\in \mathbb{R}^{n_1\times n_2 \times 3}$, $\|\mathcal{C} \|_{\infty} =\mathop{\max}\limits_{m,n,t}|C_{mnt}|$.
\\ 

From  Fig. \ref{psnr}, we see that the performance of \texttt{RkA-RBT} is comparable to  Fig. 2 in \cite{qinzhanglp2022}. Moreover, Table \ref{table1} shows that the using of the Ht-SVD is much faster than Qt-SVD in \cite{qinzhanglp2022}, which used the algorithm provided by Jia, Ng and Song \cite{ng} in MATLAB.   

 \begin{figure}
		\begin{center}
			\includegraphics[width=\linewidth]{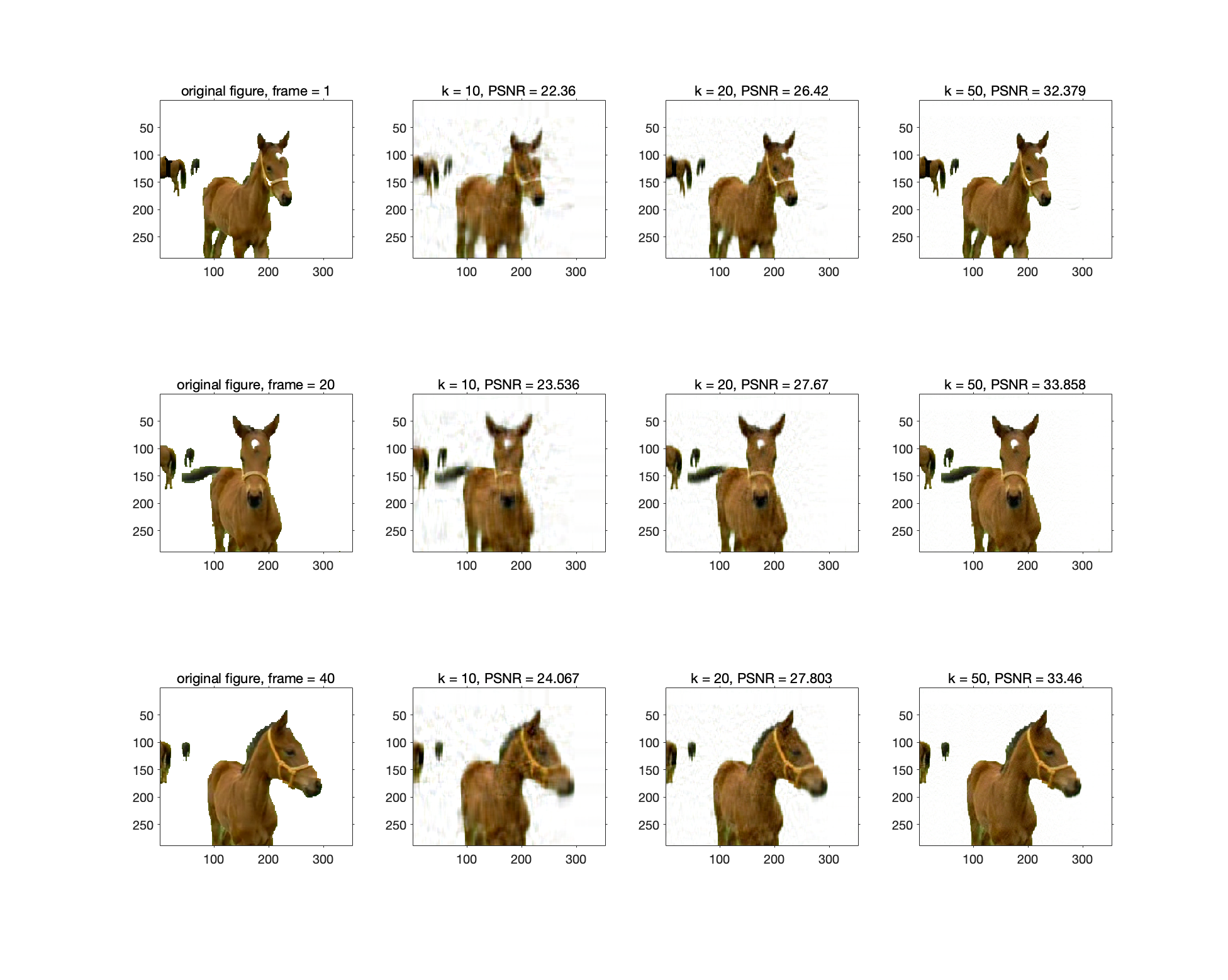}
		\end{center}
		\vskip-15pt
		\caption{PSNRs of the rank-k (k=10, 20, 50) approximations to DO01\_013 by Ht-SVD.}\label{psnr}
	\end{figure}%

\vspace{20.4cm}

\begin{table}
\begin{center}
\begin{tabular}{ |c| c| c| c| c|}
 \hline
 frame & 10 & 20 & 30 & 40\\
 \hline
 Qt-SVD & 20.861257s & 46.294602s & 89.366687s & 116.283031s\\
 \hline
 Ht-SVD & 2.715995s & 6.818766s & 9.597188s & 13.1181154s\\
\hline\end{tabular}
\\[3mm]
\caption{The running time of  Ht-SVD and Qt-SVD for the different number of frames.}
\label{table1}
\end{center}
\end{table}

\noindent \textbf{Example 6.3: (Color video deblurring)}
Video deblurring is a crucial task in video processing with applications ranging from video enhancement to computer vision tasks, like object recognition and tracking. Given an original video \( \mathcal{A} \in \mathbb{H}_c^{n_1 \times n_2 \times n_3} \) and its blurred video \( \mathcal{B} \in \mathbb{H}_c^{n_1 \times n_2 \times n_3} \), the goal is to find the deblurring filter \( \mathcal{F} \) that minimizes the reconstruction error between the filtered video \( \mathcal{F} \ast_{Ht} \mathcal{B} \) and the original video \( \mathcal{A} \).

This problem can be formulated as the following least-squares problem:

\[
\min_{\mathcal{F}} \| \mathcal{A} - \mathcal{F} \ast_{Ht} \mathcal{B} \|_F^2.
\]
Once the deblurring filter \( \mathcal{F} \) is obtained, it can be applied to a blurred video to achieve deblurring. Specifically, for a new blurred video \( \mathcal{B}_{\text{new}} \), the deblurred video \( \mathcal{A}_{\text{deblurred}} \) can be obtained by:

\[
\mathcal{A}_{\text{deblurred}} = \mathcal{F} \ast_{Ht} \mathcal{B}_{\text{new}}.
\]
In this problem, the filter \( \mathcal{F} \) is determined via a least-squares method. According to Theorem \ref{HtLSMP2}, the optimal solution for \( \mathcal{F} \) is given by \( \mathcal{F} = \mathcal{A} \ast_{Ht} \mathcal{B}^\dagger \). 
The detailed solution process is summarized in Algorithm 3.

\begin{algorithm}
	\renewcommand{\algorithmicrequire}{\textbf{Input:}}
	\renewcommand{\algorithmicensure}{\textbf{Output:}}
	\caption{{\bf (LSD-RBT)} Least-Squares Deblurring via Reduced Biquaternion Tensors}
	\label{deblur_algo}
	\begin{algorithmic}[1]
		\STATE $\textbf{Input}$: $\mathcal{A}_{(c),1},\mathcal{A}_{(c),2}\in \mathbb{C}^{n_1\times n_2 \times n_3}$ from given $\mathcal{A} =\mathcal{A}_{(c),1} e_1+ \mathcal{A}_{(c),2} e_2 \in \mathbb{H}_c^{n_1\times n_2 \times n_3}$ ; \\
  $\mathcal{B}_{(c),1},\mathcal{B}_{(c),2}\in \mathbb{C}^{n_1\times n_2 \times n_3}$ from given $\mathcal{B} =\mathcal{B}_{(c),1} e_1+ \mathcal{B}_{(c),2} e_2 \in \mathbb{H}_c^{n_1\times n_2 \times n_3}$ 
		\STATE  $\textbf{Output}$: $\mathcal{F}\in \mathbb{H}_c^{n_1 \times n_1  \times n_3}$ such that
\[
\mathcal{F} = \min_{\mathcal{F}} \| \mathcal{A} - \mathcal{F} \ast_{Ht} \mathcal{B} \|_F^2
\]
           
\STATE   $\hat{\mathcal{A}}_{(c),1}$=fft$( \mathcal{A}_{(c),1},[],3)$;$\hat{\mathcal{A}}_{(c),2}$=fft$( \mathcal{A}_{(c),2},[],3); \hat{\mathcal{B}}_{(c),1}$=fft$( \mathcal{B}_{(c),1},[],3)$;$\hat{\mathcal{B}}_{(c),2}$=fft$( \mathcal{B}_{(c),2},[],3)$

\STATE  $\textbf{for}$ $i=1,...n_{3},$ $\textbf{do}$\\
 $[\hat{U}_1^{(i)}, \hat{S}_1^{(i)}, \hat{V}_1^{(i)}]=svd(\hat{\mathcal{B}}_{(c),1}^{(i)})$,  $[\hat{U}_2^{(i)}, \hat{S}_2^{(i)}, \hat{V}_2^{(i)}]=svd(\hat{\mathcal{B}}_{(c),2}^{(i)})$\\
 $\hat{\mathcal{B}}_{(c),1}^{\dagger(i)} = \hat{V}_1^{(i)} \hat{S}_1^{\dagger(i)}  \hat{U}_1^{(i)\ast} $,
 $\hat{\mathcal{B}}_{(c),2}^{\dagger(i)} = \hat{V}_2^{(i)} \hat{S}_2^{\dagger(i)}  \hat{U}_2^{(i)\ast} $ \\
 $\hat{\mathcal{F}}_{(c),1} = \hat{\mathcal{A}}_{(c),1}^{(i)} \hat{\mathcal{B}}_{(c),1}^{\dagger(i)}$, 
  $\hat{\mathcal{F}}_{(c),2} = \hat{\mathcal{A}}_{(c),2}^{(i)} \hat{\mathcal{B}}_{(c),2}^{\dagger(i)}$ \\
 $\textbf{end for}$
  \STATE ${\mathcal{F}}_{(c),1}$=ifft($\hat{\mathcal{F}}_{(c),1}$,$\left[\,\right]$,3), ${\mathcal{F}}_{(c),2}$=ifft($\hat{\mathcal{F}}_{(c),2}$,$\left[\,\right]$,3),
  $\mathcal{F} = {\mathcal{F}}_{(c),1} e_1 + {\mathcal{F}}_{(c),2} e_2$
\STATE $\textbf{return}$   $\mathcal{F}$
	\end{algorithmic}  
\end{algorithm}

For video deblurring experiments, we select color videos from the DAVIS 2017 dataset \cite{Pont-Tuset2017}. The experimental comparison is conducted between traditional quaternion-based method \cite{Li2023} and the newly proposed reduced biquaternion method. Table \ref{table2} illustrates the comparative results in terms of PSNR, relative error, and CPU time per frame. The relative error between the deblurred video  \(\tilde{\mathcal{A}}\) and the original video  \(\mathcal{A}\) 
 is computed as follows:
\[
\text{Relative Error} = \frac{\| \mathcal{A} - \tilde{\mathcal{A}} \|_F}{\| \mathcal{A} \|_F},
\]
and PSNR is defined by (\ref{psnr2}).
As shown in Table \ref{table2}, our reduced biquaternion method \texttt{LSD-RDT} demonstrates superior performance compared to traditional quaternion-based method. Fig. \ref{fig2} and Fig. \ref{fig3} showcase the enhanced deblurring effects on the 'Tram' and 'Butterfly' videos, respectively, illustrating the effectiveness of the RB method in restoring video clarity.

\begin{table}[H]
\centering
\begin{tabular}{ccccc}
\hline
Video                      & Method               & PSNR             & Relative error    & CPU time (second) \\ \hline
\multirow{2}{*}{Tram}      & Quaternion           & 30.4592          & 0.058439          & 1.867             \\
                           & Reduced biquaternion & \textbf{33.8519} & \textbf{0.039519} & \textbf{0.152}    \\ \hline
\multirow{2}{*}{Butterfly} & Quaternion           & 32.9166          & 0.053911          & 1.947             \\
                           & Reduced biquaternion & \textbf{38.8111} & \textbf{0.027293} & \textbf{0.147}    \\ \hline
\end{tabular}
\caption{Performance Comparison of Quaternion and Reduced Biquaternion Methods for Video Deblurring.}\label{table2}
\end{table}

 \begin{figure}[H]
		\begin{center}
			\includegraphics[width=\linewidth	]{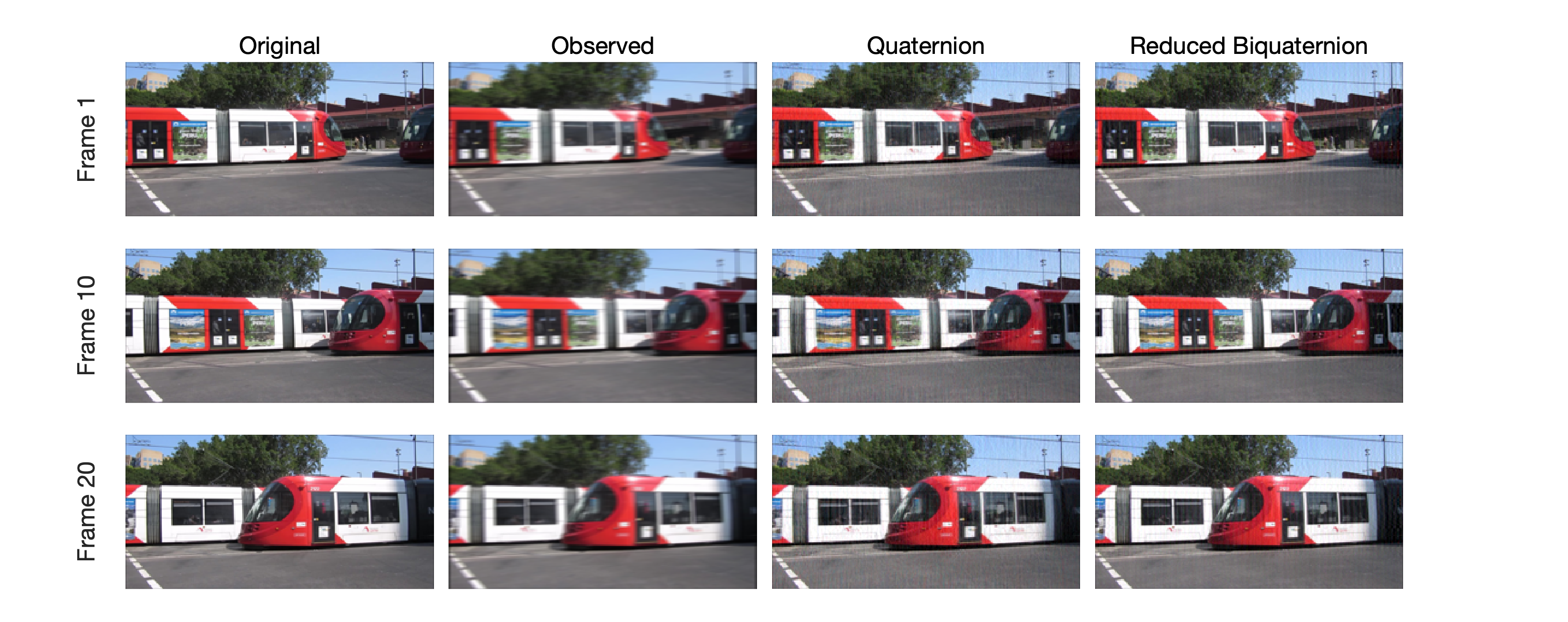}
		\end{center}
		\vskip-15pt
		\caption{Example of Deblurring Estimation for Three Blurred Frames in Video 'Tram'.}\label{fig2}
	\end{figure}%

 \begin{figure}[H]
		\begin{center}
			\includegraphics[width=\linewidth	]{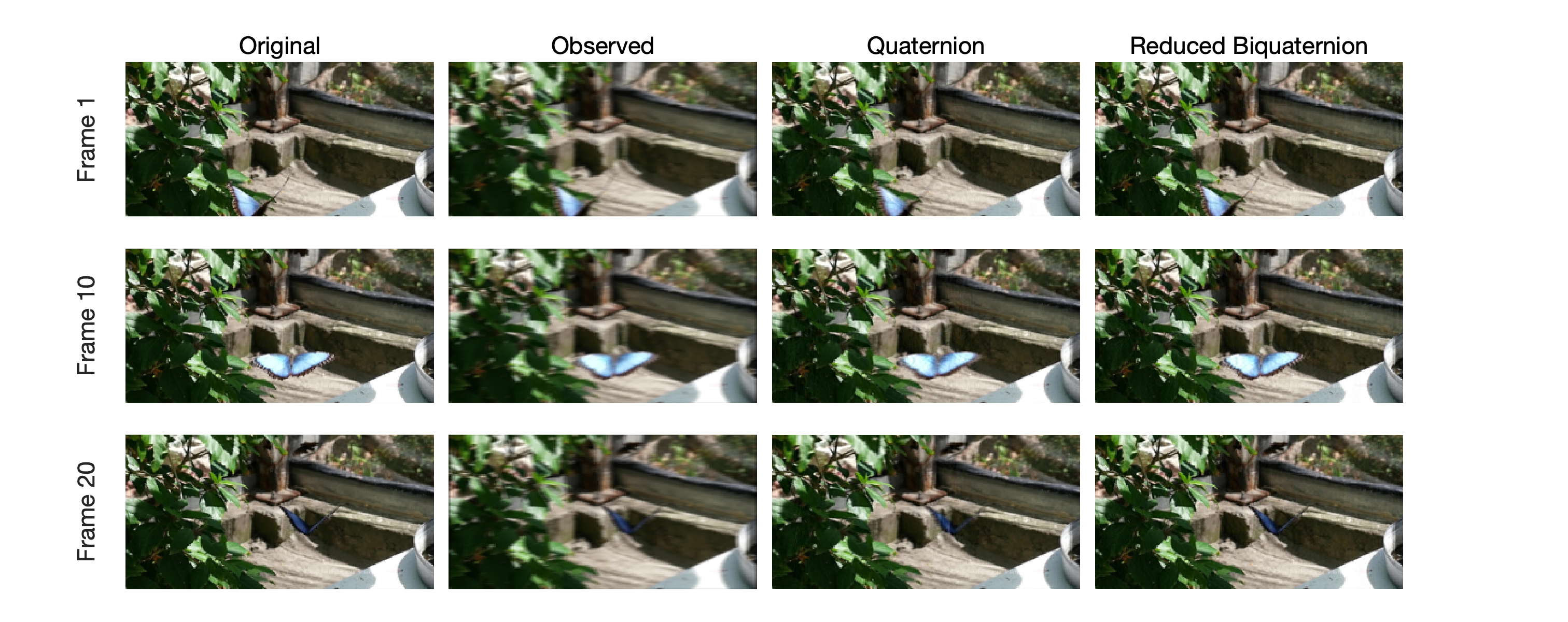}
		\end{center}
		\vskip-15pt
		\caption{Example of Deblurring Estimation for Three Blurred Frames in Video 'Butterfly'.}\label{fig3}
	\end{figure}%
\noindent \textbf{Conclusion:}  In this paper, we first define Ht-product for the third-order RB tensors. Then under this product, we define the Ht-SVD for a third-order RB tensor and develop a computing  algorithm. Moreover, we extend the Moore-Penrose inverse to the third-order RB tensors. Our Ht-SVD allows us to prove many properties of this MP inverse efficiently. Furthermore, we discuss the general solution/the least-squares solutions/minimal norm least-squares solution to the classic RB tensor equations   
via our MP inverses. As the applications of our results, we develop two algorithms and apply them to  color video  compression and deblurring. Both 
 of two experiments illustrate the superiority of our methods by comparing with other methods.   

 \section{Declarations}
{\bf Conflict of Interest:} The authors have not disclosed any competing interests.
\\
\\
{\bf Data Availability Statement:}
The data supporting the findings of this study are not publicly available due to privacy or ethical restrictions. However, interested researchers may request access to the data by contacting the corresponding author and completing any necessary data sharing agreements.

\end{document}